\documentclass{amsart}

\usepackage{amsmath,amsthm,amssymb,bm}
\usepackage{hyperref}
\usepackage{graphicx,color}
\usepackage{tikz}
\usetikzlibrary{decorations.markings}
\usetikzlibrary{tikzmark}
\usetikzlibrary{arrows.meta}
\usepackage{ytableau}
\usepackage{cases}
\usepackage{subcaption}
\usepackage{hhline}
\numberwithin{equation}{section}

\newtheorem{thm}{Theorem}[section]
\newtheorem{lem}[thm]{Lemma}
\newtheorem{prop}[thm]{Proposition}
\newtheorem{cor}[thm]{Corollary}

\theoremstyle{definition}

\newtheorem{defn}[thm]{Definition}

\newtheorem{problem}[thm]{Problem}
\newtheorem{remark}[thm]{Remark}
\newtheorem*{remark*}{Remark}

\newcommand\Par{\operatorname{Par}}

\newcommand\GPar{\operatorname{GPar}}

\newcommand\MMSVT{\operatorname{MMSVT}}

\newcommand\SSYT{\operatorname{SSYT}}
\newcommand\RSSYT{\operatorname{RSSYT}}

\newcommand\qand{\quad\mbox{and}\quad}
\newcommand\NN{\mathbb{N}}
\newcommand\QQ{\mathbb{Q}}

\newcommand\vx{\bm{x}}
\newcommand\vp{\bm{p}}
\newcommand\va{\bm{\alpha}}
\newcommand\vb{\bm{\beta}}

\newcommand\vy{\bm{y}}

\newcommand\tg{\tilde{g}}

\newcommand\ZZ{\mathbb{Z}}

\newcommand\wt{\operatorname{wt}}

\newcommand\IET{{\operatorname{IET}}}
\newcommand\ET{{\operatorname{ET}}}

\newcommand\MRPP{{\operatorname{MRPP}}}

\newcommand\lm{{\lambda/\mu}}
\newcommand\ml{{\mu/\lambda}}
\newcommand\lmc{{\lambda'/\mu'}}

\newcommand\U{--++(0,1)}

\newcommand\LStep[1]{node[above left]{\(#1\)}--++(-1,0)}

\title{Refined canonical stable Grothendieck polynomials and their duals, Part 1}

\author{Byung-Hak Hwang}
\address{Department of Mathematics, Sungkyunkwan University, Suwon,
  South Korea}
\email{byunghakhwang@gmail.com}
\author{Jihyeug Jang}
\address{Department of Mathematics, Sungkyunkwan University, Suwon,
  South Korea}
\email{4242ab@gmail.com}
\author{Jang Soo Kim}
\address{Department of Mathematics, Sungkyunkwan University, Suwon,
  South Korea}
\email{jangsookim@skku.edu}
\author{Minho Song}
\address{Department of Mathematics, Sungkyunkwan University, Suwon,
  South Korea}
\email{smh3227@skku.edu}
\author{U-Keun Song}
\address{Department of Mathematics, Sungkyunkwan University, Suwon,
  South Korea}
\email{sukeun319@gmail.com}

\date{\today}

\begin{document}

\begin{abstract}
  In this paper we introduce refined canonical stable Grothendieck polynomials
  and their duals with two infinite sequences of parameters. These
  polynomials unify several generalizations of Grothendieck polynomials
  including canonical stable Grothendieck polynomials due to Yeliussizov,
  refined Grothendieck polynomials due to Chan and Pflueger, and refined dual
  Grothendieck polynomials due to Galashin, Liu, and Grinberg. We give
  Jacobi--Trudi-like formulas, combinatorial models, Schur expansions, Schur
  positivity, and dualities of these polynomials.
\end{abstract}

\maketitle

\section{Introduction}
\label{sec:introduction}

\subsection{History of Grothendieck polynomials}
\label{sec:hist-groth-polyn}

Grothendieck polynomials were introduced
by Lascoux and Sch\"{u}tzenberger \cite{LS1982}
for studying the Grothendieck ring of vector bundles on a flag variety.
Therefore, these polynomials can be regarded as
\( K \)-theoretic analogues of Schubert polynomials.
As in the Schubert polynomial case, they are indexed by permutations,
and the stable limit of Grothendieck polynomials are symmetric functions.
In \cite{FK1996}, Fomin and Kirillov introduced \( \beta
\)-Grothendieck polynomials with a parameter \( \beta \),
which reduce to Schubert polynomials and Grothendieck polynomials when \( \beta=0 \)
and \( \beta=-1 \), respectively.
In other words, they introduced a grading on Grothendieck polynomials
such that the degree 0 parts of the graded polynomials coincide
with Schubert polynomials.
They also studied the stable Grothendieck polynomials with
the parameter \( \beta \).

When focusing on the Grothendieck ring of vector bundles on Grassmannians rather
than that of flag varieties, the stable Grothendieck polynomials are indexed by
partitions. These stable Grothendieck polynomials \( G_\lambda(\vx) \) are also called symmetric
Grothendieck polynomials in the literature and they form a basis for the
(connective) \( K \)-theory ring of Grassmannians, see
\cite[Section~2.3]{Monical2020} and references therein. In this context,
Lenart~\cite{Lenart2000} gave the Schur expansion of \( G_\lambda(\vx) \) and
showed that \( \{G_\lambda(\vx) : \mbox{\( \lambda \) is a partition}\} \) is a
basis for (a completion of) the space of symmetric functions. Lenart's result
gives a combinatorial model for \( G_\lambda(\vx) \) as pairs of tableaux.
Buch \cite{Buch2002} showed that 
the stable Grothendieck polynomial \( G_\lambda(\vx) \) is equal to a generating
function for semistandard set-valued tableaux of shape \( \lambda \),
which are reminiscent of semistandard Young tableaux, the combinatorial model for Schur functions.

After Buch's work~\cite{Buch2002}, many generalizations of \( G_\lambda(\vx) \)
have been studied from various viewpoints.
In \cite{LP2007}, Lam and Pylyavskyy introduced
the dual stable Grothendieck polynomials \( g_\lambda(\vx) \),
and found a combinatorial interpretation for them
in terms of reverse plane partitions.
The stable Grothendieck polynomials \( G_\lambda(\vx) \)
and the dual stable Grothendieck polynomials \( g_\lambda(\vx) \)
are dual with respect to the Hall inner product.
Similar to the Grothendieck polynomial case, the dual Grothendieck polynomials
form a basis for the space of symmetric functions.

The Schur functions \( s_\lambda(\vx) \) satisfy \(
\omega(s_\lambda(\vx))=s_{\lambda'}(\vx) \), where \( \omega \) is the standard
involution on symmetric functions and \( \lambda' \) is the transpose of \( \lambda \).
In search of a deformation of Grothendieck polynomials with this property,
Yeliussizov \cite{Yeliussizov2017} introduced
the canonical stable Grothendieck polynomials
\( G^{(\alpha,\beta)}_\lambda(\vx_n) \) with two parameters \( \alpha \) and \(
\beta \) defined by  
\begin{equation}\label{eq:yeli}
  G^{(\alpha,\beta)}_\lambda(\vx_n) 
  = \frac{\det
    \left(x_j^{\lambda_i+n-i}(1+\beta x_j)^{i-1}(1-\alpha x_j)^{-\lambda_i}\right)_{1\le i,j\le n}}
    {\prod_{1\le i<j\le n}(x_i-x_j)},
\end{equation}
where \( \vx_n = (x_1,\dots,x_n) \) is a sequence of \( n \) indeterminants,
and defined \( G^{(\alpha,\beta)}_\lambda(\vx) \)
to be the limit of \( G^{(\alpha,\beta)}_\lambda(\vx_n) \) as \( n\to\infty \).
When \( \alpha=0 \), the canonical stable Grothendieck polynomial \( G^{(0,\beta)}
_\lambda(\vx) \) reduces to the stable \( \beta \)-Grothendieck polynomial
indexed by \( \lambda \).
He also defined the dual canonical stable Grothendieck polynomials
\( g^{(\alpha,\beta)}_\lambda(\vx) \) by the relation
\[
  \left\langle G^{(-\alpha,-\beta)}_\mu(\vx), g^{(\alpha,\beta)}_\lambda(\vx) \right\rangle
  = \delta_{\lambda,\mu},
\]
where \( \langle -,- \rangle \) is the Hall inner product.
Yeliussizov \cite{Yeliussizov2017} showed that these polynomials
behave nicely as Schur functions under the involution \( \omega \):
\[
  \omega (G^{(\alpha,\beta)}_\lambda(\vx)) = G^{(\beta,\alpha)}_{\lambda'}(\vx)
  \qand
  \omega (g^{(\alpha,\beta)}_\lambda(\vx)) = g^{(\beta,\alpha)}_{\lambda'}(\vx).
\]
Moreover, he found combinatorial interpretations for
\( G^{(\alpha,\beta)}_\lambda(\vx) \) and \( g^{(\alpha,\beta)}_\lambda(\vx) \)
using hook-valued tableaux and rim border tableaux, respectively.

There are also generalizations of Grothendieck polynomials and their duals with
infinite parameters.
Motivated by a geometric result in \cite{Chan2021},
Chan and Pflueger \cite{CP21:grothendieck} introduced
the refined stable Grothendieck polynomials \( RG_\lambda(\vx;\vb) \)
with parameters \( \vb=(\beta_1,\beta_2,\dots) \).
One can think of these new parameters \( \vb \) as a refinement
of the parameter \( \beta \) introduced by Fomin and Kirillov.
By introducing a statistic on reverse plane partitions,
Galashin, Grinberg, and Liu \cite{GGL2016} independently defined the refined
dual stable Grothendieck polynomials \( \tg_\lambda(\vx;\vb) \).
These polynomials \( \tg_\lambda(\vx;\vb) \) are also symmetric in \( \vx \),
and generalize both Schur functions and the dual stable Grothendieck polynomials.
Although a geometric meaning of these refined parameters has been unknown,
the refined dual Grothendieck polynomials have a natural interpretation
from probability theory; see \cite{Motegi2021}.
Grinberg showed that
\( RG_\lambda(\vx;\vb) \) and \( \tg_\lambda(\vx; \vb) \) are dual with respect
to the Hall inner product; see \cite[Remark~3.9]{CP21:grothendieck}.
This fact is rather surprising since these two polynomials are introduced
in quite different contexts.

Note that \( RG_\lambda(\vx;\vb) \) and \( \tg_\lambda(\vx;\vb) \) are
multi-parameter refinements of \( G^{(0,\beta)}_\lambda(\vx) \) and \(
g^{(0,\beta)}_\lambda(\vx) \), respectively, in the sense that \(
RG_\lambda(\vx;(-\beta,-\beta,\dots))=G^{(0,\beta)}_\lambda(\vx) \) and \(
\tg_\lambda(\vx;(\beta,\beta,\dots))=g^{(0,\beta)}_\lambda(\vx) \). Hence
a natural question is to find multi-parameter refinements of \(
G^{(\alpha,\beta)}_\lambda(\vx) \) and \( g^{(\alpha,\beta)}_\lambda(\vx) \) by
introducing two infinite sequences \( \va=(\alpha_1,\alpha_2,\dots) \) and
\( \vb=(\beta_1,\beta_2,\dots) \) of parameters.
This is the main objective of this paper.

Since Grothendieck polynomials and their duals are analogues of Schur functions,
one may ask what properties of Schur functions have analogues
for Grothendieck polynomials.
For instance, the Jacobi--Trudi formula for Schur functions is well known:
\begin{equation}\label{eq:JT}
  s_\lambda(\vx)
      = \det \left( h_{\lambda_i-i+j}(\vx) \right)_{i,j=1}^{\ell(\lambda)},
\end{equation}
where \( \ell(\lambda) \) denotes the number of parts of a partition
\( \lambda \) and \( h_k(\vx) \) is the complete homogeneous symmetric function.
(For explicit definitions, see Section~\ref{sec:symmetric_function}.)
Various Jacobi--Trudi-like formulas for generalizations of Grothendieck polynomials
have been studied in the literature;
see, e.g., \cite{Amanov2022,Kim2021a,Kim_JT22,Kirillov_2016, Matsumura_2017,
  Matsumura_2018,Motegi2021,ShimozonoZabrocki,Yeliussizov2017}.

Besides the Jacobi--Trudi-like formula, Grothendieck polynomials have been studied
from various points of view.
In particular, it is known that Grothendieck polynomials have
integrable vertex models~\cite{Brubaker2023, Buciumas2020, Gunna20, MS13,Motegi2021,WZ19},
crystal structures~\cite{Galashin2017,HS20,Pan2022},
and probabilistic models~\cite{Motegi2021, yeliussizov20:_dual_groth}.

In this paper, motivated by these previous works, we introduce refined canonical
stable Grothendieck polynomials \( G_\lambda(\vx_n;\va,\vb) \) and their duals
\( g_\lambda(\vx_n;\va,\vb) \) with infinite parameters \(
\va=(\alpha_1,\alpha_2,\dots) \) and \( \vb=(\beta_1,\beta_2,\dots) \), and
furthermore their skew versions. We give
Jacobi--Trudi-like formulas, combinatorial models, Schur expansions, 
Schur positivity, and dualities of these polynomials. Our generalization unifies all 
generalizations of Grothendieck polynomials mentioned above, and the
Jacobi--Trudi-like formulas in this paper generalize those in the literature.

\subsection{Main results}
\label{sec:main-results}

Let \( \vx = (x_1,x_2,\dots) \), \( \vx_n = (x_1,\dots,x_n) \) be sequences of variables,
and \( \va = (\alpha_1,\alpha_2,\dots) \), \( \vb = (\beta_1,\beta_2,\dots) \)
sequences of parameters.

We generalize the canonical stable Grothendieck polynomials in \eqref{eq:yeli}
as follows.

\begin{defn}\label{def:bi_alt_G}
  For a partition \( \lambda \) with at most \( n \) parts, 
  the \emph{refined canonical stable Grothendieck polynomial}
  \( G_\lambda(\vx_n;\va,\vb) \) is defined by
  \begin{equation}\label{eq:defG}
    G_\lambda(\vx_n;\va,\vb) 
    = \frac{\det
      \left(x_j^{\lambda_i+n-i}
      \dfrac{(1-\beta_1 x_j)\cdots (1-\beta_{i-1} x_j)}
            {(1-\alpha_1 x_j)\cdots (1-\alpha_{\lambda_i} x_j)}
      \right)_{1\le i,j\le n}}
      {\prod_{1\le i<j\le n}(x_i-x_j)}.
  \end{equation}
\end{defn}

Note that the numerator on the right hand side of \eqref{eq:defG} is a rational
function which vanishes if \( x_i=x_j \) for any \( 1\le i<j\le n \). Hence the
right hand side of \eqref{eq:defG} can be written as a rational function whose
denominator is a product of factors of the form \( 1-\alpha_ix_j \), which in
turn can be expanded as geometric series. This shows that \(
G_\lambda(\vx_n;\va,\vb) \) in \eqref{eq:defG} is a well-defined formal power series.

Note that \( G_\lambda(\vx_n;\va,\vb) \) is a formal power series in \( \vx_n \)
rather than a polynomial in \( \vx_n \). By abuse of terminology, we will call this
a polynomial. In the literature, Grothendieck polynomials are also
called Grothendieck functions.

Most formulas in this paper are expressed using plethystic substitution, see
Section~\ref{sec:plethystic} for more explanations on this notion. The following
identity is well known in the theory of symmetric functions: for any formal
power series \( Y \) and \( Z \),
\[
  h_n[Y+Z] = \sum_{a+b=n} h_a[Y] h_b[Z].
\]
Imitating this identity, we define
\[
  h_n[Y\ominus Z] = \sum_{a-b=n} h_a[Y] h_b[Z].
\]
Using the notation \( \ominus \) we can rewrite \eqref{eq:defG} as follows
(see Section~\ref{sec:plethystic} for a proof):
\begin{equation}\label{eq:defG'}
  G_\lambda(\vx_n;\va,\vb)
  = \frac{\det
    \left( h_{\lambda_i+n-i}[x_j \ominus(A_{\lambda_i}-B_{i-1})]
    \right)_{1\le i,j\le n}}
  {\prod_{1\le i<j\le n}(x_i-x_j)},
\end{equation}
where \( A_{k}=\alpha_1+\dots+\alpha_{k} \) and
\( B_{k} = \beta_1+\dots+\beta_{k} \) for \( k\ge1 \),
and \( A_{k}=B_{k}=0 \) for \( k \le 0 \).

In \cite{GGL2016, LP2007, Yeliussizov2017}, the dual Grothendieck polynomials and
their generalizations are defined using reverse plane partitions or by the
duality with Grothendieck polynomials.
We define our generalization of dual Grothendieck polynomials via
the following bialternant formula.

\begin{defn} \label{def:bi_alt_g}
  For a partition \( \lambda \) with at most \( n \) parts, 
  the \emph{refined dual canonical stable Grothendieck polynomial}
  \( g_\lambda(\vx_n;\va,\vb) \) is defined by 
  \begin{equation}\label{eq:defg}
    g_\lambda(\vx_n;\va,\vb)
    = \frac{\det(h_{\lambda_i+n-i}[x_j-A_{\lambda_i-1}+B_{i-1}])_{1\le i,j\le n}}
      {\prod_{1\le i<j\le n}(x_i-x_j)}.
  \end{equation}
\end{defn}

Amanov and Yeliussizov~\cite{Amanov2022} found a bialternant formula for
the dual Grothendieck polynomials,
which is a special case of \eqref{eq:defg}.
Since we define \( G_\lambda(\vx_n;\va,\vb) \) and
\( g_\lambda(\vx_n;\va,\vb) \) using the bialternant formulas,
they are automatically symmetric polynomials.

According to Theorem~\ref{thm:Schur_exp G,g for Par_n},
specializing \( G_\lambda(\vx_{n+1};\va,\vb) \) at \( x_{n+1}=0 \) gives
\( G_\lambda(\vx_n;\va,\vb) \).
Thus we define \( G_\lambda(\vx;\va,\vb) \) to be the unique formal power series
such that \( G_\lambda(x_1,\dots,x_n,0,0,\dots;\va,\vb)=G_\lambda(\vx_n;\va,\vb) \)
for any \( n\ge \ell(\lambda) \).
We also define \( g_\lambda(\vx;\va,\vb) \) similarly.

Using the Cauchy--Binet theorem we obtain Schur expansions for \(
G_\lambda(\vx_n;\va,\vb) \) and \( g_\lambda(\vx_n;\va,\vb) \) 
in Theorem~\ref{thm:Schur_exp G,g for Par_n} and prove the
following Jacobi--Trudi-like formulas.

\begin{thm}\label{thm:JT_ab_intro}
  For a partition \( \lambda \) with at most \( n \) parts, we have
  \begin{align*}
    G_\lambda(\vx_n;\va,\vb)
    &= \det\left(  h_{\lambda_i-i+j}[X_n\ominus (A_{\lambda_i}-B_{i-1})]\right)_{i,j=1}^n, \mbox{ and} \\ 
    g_\lambda(\vx_n;\va,\vb)
    &= \det\left(  h_{\lambda_i-i+j}[X_n-A_{\lambda_i-1}+B_{i-1}]\right)_{i,j=1}^n,
  \end{align*}
  where \( X_n=x_1+\dots+x_n \).
\end{thm}

Using the Schur expansions of \( G_\lambda(\vx;\va,\vb) \) and \( g_\lambda(\vx;
\va,\vb) \), we show the following duality, which justifies the name ``refined
dual canonical stable Grothendieck polynomial'' for \( g_\lambda(\vx;\va,\vb) \).

\begin{thm}\label{thm:dual_intro}
  For partitions \( \lambda \) and \( \mu \),
  \[
    \left\langle G_\lambda(\vx;\va,\vb), g_\mu(\vx;\va,\vb)  \right\rangle
    = \delta_{\lambda,\mu}.
  \]
\end{thm}

Using lattice paths and the Lindstr\"om--Gessel--Viennot lemma,
we find combinatorial interpretations for the coefficients in the Schur
expansions of \( G_\lambda(\vx;\va,\vb) \) and \( g_\lambda(\vx;\va,\vb) \) in
terms of tableaux; see Theorems~\ref{thm:C_lm} and \ref{thm:c_lm}. As a
consequence we obtain combinatorial interpretations for these polynomials using
pairs of tableaux; see Corollaries~\ref{cor:G_comb2} and \ref{cor:g_comb2}. We
also obtain the following positivity phenomena.

\begin{thm}\label{thm:schur_pos_intro}
  The refined canonical stable Grothendieck polynomial $G_\lambda(\vx;\va,-\vb)$
  and the refined dual canonical stable Grothendieck polynomial
  \( g_\lambda(\vx;-\va,\vb) \) are Schur-positive,
  where \(-\va=(-\alpha_1,-\alpha_2,\dots)\) and \(-\vb=(-\beta_1,-\beta_2,\dots)\).
\end{thm}

Recall that there are two combinatorial models for \( G_\lambda(\vx) \): pairs
of tableaux due to Lenart \cite{Lenart2000} and set-valued tableaux due to Buch
\cite{Buch2002}. Similarly, Lam and Pylyavskyy \cite{LP2007} found two
combinatorial models for \( g_\lambda(\vx) \). Our combinatorial models in
Corollaries~\ref{cor:G_comb2} and \ref{cor:g_comb2} generalize those of Lenart
and Lam--Pylyavskyy. In the forthcoming paper \cite{Hwang-preprint} we also give the following combinatorial interpretations
for \( G_\lambda(\vx;\va,\vb) \) and \( g_\lambda(\vx;\va,\vb) \) generalizing
Buch's and Lam--Pylyavskyy's; see \cite{Hwang-preprint} for precise definitions of undefined terms.

\begin{thm}[\cite{Hwang-preprint}]\label{thm:comb_intro}
  For a partition \( \lambda \),
  \begin{align*}
    G_\lambda(\vx;\va,\vb) &= \sum_{T\in \MMSVT(\lambda)} \wt(T),\\
    g_\lambda(\vx;\va,\vb) &= \sum_{T\in\MRPP(\lambda)} \wt(T),
  \end{align*}
  where \( \MMSVT(\lambda) \) and \( \MRPP(\lambda) \) stand
  for the set of marked multiset-valued tableaux and marked reverse plane partitions
  of shape \( \lambda \), respectively.
\end{thm}

Using the combinatorial interpretations in
Theorem~\ref{thm:comb_intro} one can naturally extend the definition
of \( G_\lambda(\vx;\va,\vb) \) and \( g_\lambda(\vx;\va,\vb) \) to
skew shapes. In our forthcoming paper \cite{Hwang-preprint}, we show
that \( G_\lm(\vx;\va,\vb) \) and \( g_\lm(\vx;\va,\vb) \) are
symmetric in \( \vx \), and provide Jacobi--Trudi-like formulas for these
polynomials; see Theorem~\ref{thm:skew_Gg_JT_formula}. As in the
canonical stable Grothendieck polynomial case, applying the involution
\( \omega \) to these polynomials can be described nicely as follows.

\begin{thm}\label{thm:omega-intro}
  For partitions \( \lambda \) and \( \mu \) with \( \mu\subseteq\lambda \),
  we have
  \begin{align*}
    \omega(G_\lm(\vx;\va,\vb)) &= G_{\lambda'/\mu'}(\vx;-\vb,-\va), \\
    \omega(g_\lm(\vx;\va,\vb)) &= g_{\lambda'/\mu'}(\vx;-\vb,-\va).
  \end{align*}
\end{thm}

Our generalizations \( G_\lambda(\vx;\va,\vb) \) and \( g_\lambda(\vx;\va,\vb) \)
generalize several well-studied variations of Grothendieck polynomials as follows.
Here \( \bm{0} = (0,0,\dots) \), \( \bm{1}=(1,1,\dots) \),
\( \va_0 = (\alpha,\alpha,\dots) \) and \( \vb_0 = (\beta,\beta,\dots) \).

\begin{center}
  \begin{tabular}{|c|c|c|}
    \hline
    Variations of Grothendieck polynomials & introduced in & how to specialize \\ \hhline{|=|=|=|}
    \( G_{\nu/\lambda}(\vx) \) & Buch~\cite{Buch2002} &
      \( G_{\nu/\lambda}(\vx;\bm{0},\bm{1}) \)  \\ \hline
    \( g_{\lambda/\mu}(\vx) \) & Lam--Pylyavskyy~\cite{LP2007} &
      \( g_{\lambda/\mu}(\vx;\bm{0},\bm{1}) \) \\  \hline
    \( G^{(\alpha,\beta)}_\lambda(\vx) \) & Yeliussizov~\cite{Yeliussizov2017} &
        \( G_{\lambda}(\vx;\va_0,-\vb_0) \) \\  \hline
    \( g^{(\alpha,\beta)}_\lambda(\vx) \) & Yeliussizov~\cite{Yeliussizov2017} &
        \( g_{\lambda}(\vx;-\va_0,\vb_0) \) \\  \hline
    \( RG_\sigma(\vx;\vb) \) & Chan--Pflueger~\cite{CP21:grothendieck} &
      \( G_\sigma(\vx;\bm{0},\vb) \) \\  \hline
    \( \tg_{\lambda/\mu}(\vx;\vb) \) & Galashin--Grinberg--Liu~\cite{GGL2016} &
      \( g_{\lambda/\mu}(\vx;\bm{0},\vb) \) \\  \hline
  \end{tabular}
\end{center}

\medskip

The remainder of the paper is organized as follows.
In Section~\ref{sec:preliminary}, we briefly recall
necessary definitions and results in symmetric functions.
In Section~\ref{sec:schur_expansion}, we give Schur expansions and
Jacobi--Trudi-like formulas (Theorem~\ref{thm:JT_ab_intro}) for \(
G_\lambda(\vx_n;\va,\vb) \) and \( g_\lambda(\vx_n;\va,\vb) \) using the
Cauchy--Binet theorem and prove their duality with respect to the Hall inner product.
In Section~\ref{sec:lattice}, we give tableau models for \( G_\lambda(\vx;\va,\vb)
\) and \( g_\lambda(\vx;\va,\vb) \) using lattice paths and the
Lindstr\"om--Gessel--Viennot lemma. From these combinatorial models, we obtain
Schur positivity of \( G_\lambda(\vx;\va,-\vb) \) and \( g_\lambda(\vx;-\va,\vb) \).
In Section~\ref{sec:var JT}, we modify the Jacobi--Trudi-like formulas
in Section~\ref{sec:schur_expansion}.
The modified formulas are later used to show that 
\(G_\lambda(\vx;\va,\vb) \) and \(g_\lambda(\vx;\va,\vb) \) 
defined by bialternant formulas are special cases of
their skew versions defined by combinatorial formulas.
In Section~\ref{sec:omega_involution}, we give skew Schur expansions
for \( G_{\lambda/\mu}(\vx;\va,\vb) \) and \( g_{\lambda/\mu}(\vx;\va,\vb) \).
Using these expansions, we show that the refined canonical stable
Grothendieck polynomials and their duals behave nicely under 
the involution \( \omega \).
Finally, in Section~\ref{sec:further_study}, we suggest some future directions of study.

\section{Preliminaries}
\label{sec:preliminary}

In this section, we set up notations and give the necessary background.

Throughout this paper, we denote by \( \ZZ \) (resp.~\( \NN \)) the set of integers (resp.~nonnegative integers).
For \( n\in\NN \) we write \( [n]=\{1,2,\dots,n\} \).

\subsection{Partitions and tableaux}

A \emph{partition} is a weakly decreasing sequence \(
\lambda=(\lambda_1,\dots,\lambda_\ell) \) of positive integers. Each \(
\lambda_i \) is called a \emph{part}, and the \emph{length}
\( \ell(\lambda)=\ell \) of \( \lambda \) is the number of parts in \( \lambda \).
We will also consider a partition \( \lambda \) as an infinite sequence
by appending zeros at the end:
\( \lambda=(\lambda_1,\dots,\lambda_{\ell},0,0,\dots) \).
In other words, we define \( \lambda_i=0 \) for \( i>\ell(\lambda) \).
We denote by \( \Par_n \) the set of partitions with at most \( n \) parts,
and by \( \Par \) the set of all partitions.

The \emph{Young diagram} of a partition \(\lambda=(\lambda_1,\dots,\lambda_n)\)
is the set \( \{ (i,j)\in \ZZ^2 : 1\leq i \leq n, 1 \leq j \leq \lambda_i \} \). 
We will identify  a partition \(\lambda\) with its Young diagram.
Therefore a partition \( \lambda \) is considered as a sequence \(
(\lambda_1,\dots,\lambda_n) \) of integers 
and also as a set of pairs \( (i,j) \) of integers.
Each element \((i,j) \in \lambda\) is called a \emph{cell} of \(\lambda\).
The Young diagram of \( \lambda \) is visualized by placing a square
in the \( i \)th row and the \( j \)th column for each cell \( (i,j) \)
in \( \lambda \) using the matrix coordinates.
The \emph{transpose} \( \lambda' \) of \( \lambda \) is the partition
whose Young diagram is obtained from that of \( \lambda \)
by reflecting across the diagonal.
The \emph{content} \( c(i,j) \) of a cell \( (i,j)\in\lambda \) is \( j-i \).
See Figure~\ref{fig:yd}.

\begin{figure}
  \begin{ytableau}
    0 & 1 & 2 &  3  \\
    -1 & 0 & 1 \\
    -2 
  \end{ytableau} \qquad \qquad
  \begin{ytableau}
    0 & 1 & 2   \\
    -1 & 0  \\
    -2 & -1 \\
    -3
  \end{ytableau}
  \caption{The Young diagram of \( \lambda=(4,3,1) \) on the left and that of its transpose
    \( \lambda'=(3,2,2,1) \) on the right.
  The integer in each cell indicates the content of the cell.}
\label{fig:yd}
\end{figure}

For two partitions \(\lambda,\mu \in \Par\), we write \(\mu\subseteq \lambda\)
if \(\mu\) is contained in \(\lambda\) as Young diagrams. Equivalently, \(\mu
\subseteq \lambda\) if \( \ell(\mu)\le \ell(\lambda) \) and 
\(\mu_i \leq \lambda_i\) for all \(1\leq i \leq \ell(\lambda)\).
For two partitions \(\lambda,\mu \in \Par\) with \(\mu \subseteq \lambda\),
the \emph{skew shape} \(\lambda/\mu\) is defined to be
\(\lambda-\mu=\{ (i,j): 1\leq i \leq \ell(\lambda), \mu_i < j \leq \lambda_i \}\).
We denote by \( |\lm| \) the number of cells in \( \lm \).
We consider a partition \( \lambda \) as the skew shape \( \lambda/\emptyset \)
where \( \emptyset \) stands for the empty partition \( (0,0,\dots) \).

A \emph{semistandard Young tableau} of shape \(\lambda/\mu \) is
a filling \(T\) of the cells in \(\lambda/\mu\) with positive integers
such that the integers are weakly increasing in each row and
strictly increasing in each column.
Let \(\SSYT(\lambda/\mu)\) be the set of semistandard Young tableaux
of shape \(\lambda/\mu\).
See Figure~\ref{fig:ssyt}.

\begin{figure}
  \begin{ytableau}
    \none & \none & 1 & 1 & 3 \\
    \none & 1 & 3 & 4 \\
    2 & 2
  \end{ytableau}
  \caption{A semistandard Young tableau of shape \( (5,4,2)/(2,1) \).}
\label{fig:ssyt}
\end{figure}

\subsection{Symmetric functions}
\label{sec:symmetric_function}
Let \( \vx=(x_1,x_2,\dots) \) be a set of formal variables, and \( \vx_n =
(x_1,\dots,x_n) \). A \emph{symmetric polynomial} is a polynomial \(
f(\vx_n)\in\QQ[x_1,\dots,x_n] \) that is invariant under permuting the
variables. A \emph{symmetric function} is a formal power series \(
f(\vx)\in\QQ[[x_1,x_2,\dots]] \) such that \( f(\vx) \) is of bounded-degree and
invariant under permuting the variables. Let \( \Lambda \) be the \( \QQ
\)-algebra of all symmetric functions in \( \QQ[[x_1,x_2,\dots]] \). For \(
k\ge0 \), let \( \Lambda^k \) be the subspace of \( \QQ[[x_1,x_2,\dots]] \)
consisting of homogeneous symmetric functions of degree \( k \) together with \(
0 \). Then \( \Lambda=\oplus_{k\ge0}\Lambda^k \). The \emph{completion} \(
\hat{\Lambda} \) of \( \Lambda \) is the set of symmetric functions with
possibly unbounded degree. In other words, each element \( f\in \hat{\Lambda} \)
is of the form \( f=\sum_{k\ge0}f_k \), where \( f_k\in \Lambda^k \). Good
general references for the symmetric function theory are \cite{Macdonald} and
\cite[Chapter 7]{EC2}.

For \( n\ge 1 \), the \emph{complete homogeneous symmetric function}
\( h_n(\vx) \) is defined to be
\[
  h_n(\vx) = \sum_{i_1\le\dots\le i_n} x_{i_1} \cdots x_{i_n},
\]
and the \emph{elementary symmetric function} \( e_n(\vx) \) is defined to be
\[
  e_n(\vx) = \sum_{i_1<\dots<i_n} x_{i_1} \cdots x_{i_n}.
\]
By convention, we let \( h_0(\vx) = e_0(\vx) = 1 \) and \( h_n(\vx) = e_n(\vx) =
0 \) for \( n < 0 \). For a partition \( \lambda=(\lambda_1,\dots,\lambda_\ell)
\), we define \( h_{\lambda}(\vx) = h_{\lambda_1}(\vx)\cdots
h_{\lambda_\ell}(\vx) \) and \( e_{\lambda}(\vx) = e_{\lambda_1}(\vx)\cdots
e_{\lambda_\ell}(\vx) \). It is well known that \( \{ h_\lambda(\vx):
|\lambda|=k \} \) and \( \{ e_\lambda(\vx): |\lambda|=k\} \) are bases for \(
\Lambda^k \), \( k\ge0 \). Hence each of the two families \( \{h_1(\vx),
h_2(\vx),\dots\} \) and \( \{e_1(\vx), e_2(\vx),\dots\} \) is algebraically
independent, and we have \( \Lambda = \QQ[h_1(\vx),h_2(\vx),\dots] =
\QQ[e_1(\vx),e_2(\vx),\dots] \). Define an endomorphism \( \omega \) of the \(
\QQ \)-algebra \( \Lambda \) given by
\begin{align*} 
  \omega:
  \Lambda &\longrightarrow \Lambda \\
  h_n(\vx) &\longmapsto e_n(\vx).
\end{align*}
It is well known that 
the map \( \omega \) is an involution, i.e., \( \omega^2 = \operatorname{id}_\Lambda \).

Another important basis is the Schur function basis.
For a partition \( \lambda\in\Par_n \), the \emph{Schur polynomial}
\( s_\lambda(\vx_n) \) of shape \( \lambda \) is defined as follows:
\[
  s_\lambda(\vx_n)
  = \frac{\det (x_j^{\lambda_i+n-i})_{1\le i,j\le n}}
    {\prod_{1\le i<j\le n}(x_i-x_j)}.
\]
The \emph{Schur function} \( s_\lambda(\vx) \) is defined to be the
coefficient-wise limit of \( s_\lambda(\vx_n) \) as \( n\to\infty \). By
definition, Schur functions are symmetric, and it is known that \(
\{s_\lambda(\vx):|\lambda|=k\} \) is a basis for \( \Lambda^k \), \( k\ge0 \).
They have a combinatorial description which naturally extends to skew shapes as
follows. For partitions \( \lambda \) and \( \mu \) with \( \mu\subseteq\lambda
\), the Schur function \( s_{\lambda/\mu}(\vx) \) of shape \( \lambda/\mu \) is
defined to be the generating function for semistandard Young tableaux of shape
\( \lambda/\mu \):
\[
  s_{\lambda/\mu}(\vx) = \sum_{T\in\SSYT(\lambda/\mu)} \vx^T,
\]
where \( \vx^T=x_1^{m_1} x_2^{m_2} \cdots \) and
\( m_i \) is the number of appearances of \( i \) in \( T \).
Schur functions of skew shapes are also symmetric.

The \emph{Hall inner product} is the inner product on \( \Lambda \) defined by
\[
  \langle s_\lambda(\vx), s_\mu(\vx) \rangle = \delta_{\lambda,\mu}.
\]
This inner product is naturally extended to \( \hat{\Lambda}\times \Lambda \)
because every element \( f\in \hat{\Lambda} \) can be written as \(
f=\sum_{\lambda\in \Par}c_\lambda s_\lambda(\vx) \). Note, however, that it
cannot be extended to \( \hat{\Lambda}\times \hat{\Lambda} \) in the natural way
since, for instance, \( \langle \sum_{n\ge0} s_{(n)}(\vx) , \sum_{n\ge0} (-1)^n
s_{(n)}(\vx) \rangle =\sum_{n\ge0}(-1)^n \) does not converge.

The Jacobi--Trudi formula and its dual give expansions of Schur functions
in terms of \( h_n(\vx) \) and \( e_n(\vx) \) as determinants:
for partitions \( \mu\subseteq\lambda \),
\begin{align*}
  s_{\lambda/\mu}(\vx)
  &= \det \left( h_{\lambda_i-\mu_j-i+j}(\vx) \right)_{i,j=1}^{\ell(\lambda)},
  \mbox{ and} \\
  s_{\lambda'/\mu'}(\vx)
  &= \det \left( e_{\lambda_i-\mu_j-i+j}(\vx) \right)_{i,j=1}^{\ell(\lambda)}.
\end{align*}
Since the involution \( \omega \) is an algebra homomorphism, we have
\[
  \omega \left( s_{\lambda/\mu}(\vx) \right) = s_{\lambda'/\mu'}(\vx).
\]

In this paper, we consider the completion \( \hat{\Lambda}_{\QQ[\va,\vb]} \)
of \( \Lambda_{\QQ[\va,\vb]}=\QQ[\va,\vb]\otimes\Lambda \)
instead of \( \Lambda \), because the stable Grothendieck polynomials are
unbounded-degree symmetric functions, and our generalizations, \(
G_\lambda(\vx;\va,\vb) \) and \( g_\lambda(\vx;\va,\vb) \), need countably many
additional parameters \( \alpha_1, \alpha_2, \dots, \beta_1, \beta_2, \dots \).
Moreover, we lift the Hall inner product from \( \hat{\Lambda}\times\Lambda \)
to \( \hat{\Lambda}_{\QQ[\va,\vb]}\times\Lambda_{\QQ[\va,\vb]} \) via
\( \QQ[\va,\vb] \)-bilinearity.

\subsection{Plethystic substitutions}
\label{sec:plethystic}
We define plethystic substitution via the power sum symmetric functions.
For \( n\ge 1 \),
the \emph{power sum symmetric function} \( p_n(\vx) \) is defined to be
\[
  p_n(\vx) = x_1^n+ x_2^n+\cdots,
\]
and \( p_\lambda(\vx) := p_{\lambda_1}(\vx)\cdots p_{\lambda_\ell}(\vx) \) for a
partition \( \lambda=(\lambda_1,\dots,\lambda_\ell) \) with \(
\ell(\lambda)=\ell\ge1 \) and \( p_\emptyset(\vx):=1 \). Then \( \{
p_\lambda(\vx) \}_{\lambda\in\Par} \) is also a basis for \( \Lambda \), and \(
\Lambda=\QQ[p_1(\vx),p_2(\vx),\dots] \).

For a formal power series \( Z\in\QQ[[z_1,z_2,\dots]] \),
the \emph{plethystic substitution \( p_n[Z] \) of \( Z \) into \( p_n(\vx) \)} is
defined to be the formal power series obtained from \( Z \)
by replacing each \( z_i \) by \( z_i^n \).
Since \( \{p_n(\vx)\}_{n\ge 1} \) are algebraically independent,
and generate \( \Lambda \), 
we can extend \( p_n[Z] \) to the plethystic substitution \( f[Z] \)
into any symmetric function \( f(\vx) \).
By definition, 
\begin{equation}
  \label{eq:f[z]=f(z)}
  f[z_1+z_2+\cdots] = f(z_1,z_2,\dots).
\end{equation}
For more properties of plethystic substitution,
we refer the reader to \cite{Loehr_2011}.

For integers \( r,s \), and \( n \), we define
\[
  X=x_1+x_2+\cdots, \qquad
  X_n=x_1+x_2+\dots+x_n, \qquad
  X_{[r,s]}=x_r+x_{r+1}+\dots+x_s,
\]
\[
  A_n = \alpha_1+\alpha_2+\dots+\alpha_n,\qquad
  B_n = \beta_1+\beta_2+\dots+\beta_n,
\]
where \( X_n=A_n=B_n=0 \) if \( n\le 0 \) and \( X_{[r,s]}=0 \) if \( r>s \).

In this paper, we only consider the plethystic substitution into
\( h_n(\vx) \) and \( e_n(\vx) \).
We will use the following well-known properties frequently,
which follows from \cite[Theorem~6]{Loehr_2011} and \cite[Proposition~2.1]{Kim_JT22}.
\begin{prop}\label{prop:h[-Z]}
  For \( n\ge 0 \) and \( Z\in\QQ[[z_1,z_2,\dots]] \), we have
  \[
    h_n[-Z] = (-1)^n e_n[Z] \qand e_n[-Z] = (-1)^n h_n[Z].
  \]
  For \( Z_1,Z_2\in\QQ[[z_1,z_2,\dots]] \),
  \[
    h_n[Z_1+Z_2] = \sum_{a+b=n} h_a[Z_1] h_b[Z_2] \qand
    e_n[Z_1+Z_2] = \sum_{a+b=n} e_a[Z_1] e_b[Z_2].
  \]
\end{prop}

The following lemma will be used in Section~\ref{sec:lattice}.

\begin{lem}[{\cite[Lemma~5.2]{Kim_JT22}}] \label{lem:h_m[Z]}
  Suppose that \( Z \) is any formal power series, \( z \) is a single
  variable, and \( c\in\{0,1\} \). Then we have
  \[
    h_m[Z] = h_m[Z-cz] + czh_{m-1}[Z].
  \]
\end{lem}

We introduce a novel notation \( \ominus \) as follows.

\begin{defn}
  For \( n\in\ZZ \) and
  \( Z_1,Z_2\in \QQ[[z_1,z_2,\dots]] \) with no constant terms,
  \[
    h_n[Z_1\ominus Z_2] = \sum_{a-b=n} h_a[Z_1] h_b[Z_2] \qand
    e_n[Z_1\ominus Z_2] = \sum_{a-b=n} e_a[Z_1] e_b[Z_2].
  \]
\end{defn}

Informally, the notation \( Z_1 \ominus Z_2 \) can be considered as \( Z_1 + Z_2
\) regarding each variable \( z_i \) used in \( Z_2 \) as a variable of degree
\( -1 \). The special case \( h_n[X\ominus X] = \sum_{k\ge0} h_{n+k}[X]h_k[X] \)
of this notation is considered in \cite{BK72} in order to express \(
\sum_{\lambda\in\Par_n}s_\lambda(\vx) \); see also \cite[Exercise 7.16]{EC2}.

We remark some properties of this notation \( Z_1\ominus Z_2 \). First, in the
sum of the above definition of \( h_n[Z_1\ominus Z_2] \)
(resp.~\( e_n[Z_1\ominus Z_2] \)) there can be infinitely many nonzero terms \(
h_a[Z_1] h_b[Z_2] \) (resp.~\(e_a[Z_1] e_b[Z_2] \)).
Second, even if \( n < 0 \), \( h_n[Z_1\ominus Z_2] \) and \( e_n[Z_1\ominus
Z_2] \) can be nonzero. For example,
\begin{equation}\label{eq:h-2}
  h_{-2}[z_1\ominus z_2] = \sum_{k\ge 0} h_{k}[z_1] h_{k+2}[z_2]
  = \sum_{k\ge 0} z_1^{k} z_2^{k+2}.
\end{equation}
We also note that, contrary to the usual plethystic substitution, \(
f[Z_1\ominus Z_2] \) cannot be defined for all symmetric functions \( f(\vx) \)
uniformly. For instance, we have \( h_1(\vx)=e_1(\vx) \) but \( h_1[z_1\ominus z_2] \ne
e_1[z_1\ominus z_2] \).

Now we show the equivalence of the two bialternant formulas
\eqref{eq:defG} and \eqref{eq:defG'} for \( G_\lambda(\vx_n;\va,\vb) \).
We first need the following lemma.

\begin{lem}\label{lem:hAB}
  For \( r,s\in\ZZ \), we have
  \[
    \sum_{k\ge0} h_k[A_r-B_s] x^k 
    =\frac{(1-\beta_1 x)\cdots(1-\beta_{s}x)}{(1-\alpha_1 x)\cdots(1-\alpha_r x)}.
  \]  
\end{lem}
\begin{proof}
  By Proposition~\ref{prop:h[-Z]}, \( h_k[A_r-B_s]=\sum_{i=0}^k h_i[A_r](-1)^{k-i} e_{k-i}[B_s] \).
Thus the left hand side of the equation is equal to 
\[
  \sum_{k\ge0} \sum_{i=0}^k h_i[A_r](-1)^{k-i} e_{k-i}[B_s] x^k
  = \sum_{i\ge0} h_i[A_r] x^{i} \sum_{j\ge0} (-1)^j e_{j}[B_s] x^{j}.
\]
 Then, by \eqref{eq:f[z]=f(z)}, the well-known identities
 \[
   \sum_{i\ge0} h_i(\alpha_1,\dots,\alpha_r)x^i = \frac{1}{(1-\alpha_1 x)\cdots(1-\alpha_r x)},
   \quad \sum_{i\ge0} e_i(\beta_1,\dots,\beta_s)x^i = (1+\beta_1
   x)\cdots(1+\beta_s x)
 \]
 complete the proof.
\end{proof}

By Lemma~\ref{lem:hAB},
we obtain that \eqref{eq:defG} and \eqref{eq:defG'} are equivalent as follows:
\begin{multline*}
  h_{\lambda_i+n-i}[x_j \ominus(A_{\lambda_i}-B_{i-1})] 
  = \sum_{k\ge0} h_{\lambda_i+n-i+k}(x_j) h_k[A_{\lambda_i}-B_{i-1}]\\
  = x_j^{\lambda_i+n-i} \sum_{k\ge0} x_j^{k} h_k[A_{\lambda_i}-B_{i-1}]
  =x_j^{\lambda_i+n-i}
  \dfrac{(1-\beta_1 x_j)\cdots (1-\beta_{i-1} x_j)}
  {(1-\alpha_1 x_j)\cdots (1-\alpha_{\lambda_i} x_j)}.
\end{multline*}

Finally, we give simple lemmas which will be used in later sections.

\begin{lem}\label{lem:det(h)=0}
  Let \( \lambda,\mu\in\Par_n \) with \( \mu\not\subseteq\lambda \)
  and let \( Z_{i,j} \) be any formal power series for \( 1\le i,j\le n \).
  Then
  \[
    \det \left( h_{\lambda_i-\mu_j-i+j}[Z_{i,j}] \right)_{i,j=1}^{n}
    =\det \left( e_{\lambda_i-\mu_j-i+j}[Z_{i,j}] \right)_{i,j=1}^{n} = 0. 
  \]
\end{lem}
\begin{proof}
  We use the same argument in \cite[Ch I, (5.7)]{Macdonald}. By assumption, we
  have \( \mu_k > \lambda_k \) for some \( 1\le k\le n \). For \( 1\le j\le k\le
  i\le n \), we have \( \mu_j\ge\mu_k>\lambda_k\ge \lambda_i \) so that \(
  \lambda_i-\mu_j-i+j < 0 \). This implies that the \( (i,j) \)-entry of the
  matrix for each determinant equals \( 0 \). Therefore, using elementary linear
  algebra, we conclude that each determinant equals 0.
\end{proof}

Note that in Lemma~\ref{lem:det(h)=0}, each \( Z_{i,j} \) must be a formal power
series and not of the form \( Z_{i,j}=A_{i,j}\ominus B_{i,j} \). Indeed, if \(
\lambda=\emptyset \), \( \mu=(2) \), \( n=1 \), and \( Z_{1,1}=z_1\ominus z_2
\), then the first determinant in this lemma is \( h_{-2}[z_1\ominus z_2]\ne0 \)
as shown in \eqref{eq:h-2}.

\begin{lem}\label{lem:det(h)=det*det}
  Let \( \lambda,\mu\in\Par_n \) with \( \mu\subseteq\lambda \)
  and let \( Z_{i,j} \) be any formal power series for \( 1\le i,j\le n \).
  If \( 1\le k\le n \) is an integer such that \( \lambda_k=\mu_k \), then
  \[
    \det \left( h_{\lambda_i-\mu_j-i+j}[Z_{i,j}] \right)_{i,j=1}^{n}
    =\det \left( h_{\lambda_i-\mu_j-i+j}[Z_{i,j}] \right)_{i,j=1}^{k-1}
    \det \left( h_{\lambda_i-\mu_j-i+j}[Z_{i,j}] \right)_{i,j=k+1}^{n}. 
  \]
\end{lem}
\begin{proof}
  If \( 1\le j \le k\le i\le n \) and \( j < i \), then 
  \( h_{\lambda_i-\mu_j-i+j}[Z_{i,j}]=0 \) because
  \[
    \lambda_i-\mu_j-i+j \le \lambda_k-\mu_k-1 =-1<0.
  \]
  Thus the matrix on the left hand side is a block triangular matrix and since
  \( h_{\lambda_i-\mu_j-i+j}[Z_{i,j}]=1 \) for \( i=j=k \), we obtain the
  identity.
\end{proof}

\begin{lem}\label{lem:det(h)=det(h)}
  Let \( \lambda,\mu\in\Par_n \) with \( \mu\subseteq\lambda \)
  and let \( Z_{i,j} \) be any formal power series for \( 1\le i,j\le n \).
  Then for any \( m\ge n \), we have
\[
  \det \left( h_{\lambda_i-\mu_j-i+j}[Z_{i,j}] \right)_{i,j=1}^{m}
  =\det \left( h_{\lambda_i-\mu_j-i+j}[Z_{i,j}] \right)_{i,j=1}^{n}. 
\]
\end{lem}
\begin{proof}
  Since \( m\ge n \), we can consider \( \lambda \) and \( \mu \) as partitions
  in \( \Par_m \) with \( \lambda_i=\mu_i=0 \) for \( n+1\le i\le m \).
  Applying Lemma~\ref{lem:det(h)=det*det} repeatedly for \( k=m,m-1,\dots,n+1 \)
  gives this lemma.
\end{proof}

\section{Schur expansions, Jacobi--Trudi-like formulas, and duality}
\label{sec:schur_expansion}

In this section, we give Schur expansions and Jacobi--Trudi-like formulas for
\(G_{\lambda}(\vx;\va,\vb)\) and \(g_{\lambda}(\vx;\va,\vb)\) using the
Cauchy--Binet theorem. We also show the duality between 
\(G_{\lambda}(\vx;\va,\vb)\) and \(g_{\lambda}(\vx;\va,\vb)\) with respect to
the Hall inner product.

We need some notation for matrices.
We write \(P=(p_{i,j})_{i\in I,j\in J}\) to mean that \(P\) is a matrix with
row index set \(I\) and column index set \(J\). 
For \(P=(p_{i,j})_{i\in I,j\in J}\), \( K\subseteq I \), and \( L\subseteq J \),
we define
\[
  P_K = (p_{i,j})_{i\in K, j\in J}, \qquad
  P^L = (p_{i,j})_{i\in I, j\in L}, \qquad
  P_K^L = (p_{i,j})_{i\in K, j\in L}.
\]
For \(\lambda,\mu\in \Par_n\), \(P=(p_{i,j})_{i\in [n],j \in \NN}\),
\(Q=(q_{i,j})_{i\in\NN,j\in [n]}\), and \(R=(r_{i,j})_{i,j\in\NN}\), we also define
\begin{align*}
  P^{\lambda}&:=(p_{i,\lambda_j+n-j})_{1\le i,j \le n}, \\
  Q_{\lambda}&:=(q_{\lambda_i+n-i,j})_{1\le i,j \le n}, \\
  R_\lambda^\mu&:=(r_{\lambda_i+n-i,\mu_j+n-j})_{1\le i,j \le n}.
\end{align*}

The following lemma is a useful tool in this paper.

\begin{lem}[Cauchy--Binet theorem for partitions]  \label{lem:C-B}
  Suppose that
  \(P=(p_{i,j})_{i\in [n],j\in \NN}\) and \(Q=(q_{i,j})_{i\in \NN,j\in [n]}\)
  are matrices such that each entry in \( PQ \) is well defined as a series.
  Then we have
  \begin{equation*}
    \det (P Q) = \sum_{\mu \in \Par_n} (\det P^{\mu}) (\det Q_{\mu}).
  \end{equation*}  
\end{lem}
\begin{proof}
  By the Cauchy--Binet theorem, we have
  \begin{equation}\label{eq:C-B}
    \det (PQ)= \sum_{I\subseteq \NN , |I|=n} (\det P^I) (\det Q_I).
  \end{equation}

For a partition $\mu\in\Par_n$, we write $[\mu]:=\{\mu_n< \mu_{n-1}+1< \dots< \mu_1+n-1\}\subseteq \NN$. One can easily check that the map
\( \mu \mapsto [\mu] \) is a bijection from \( \Par_n\) to \( \{I\subseteq \NN: |I|=n\} \).
 Then we can rewrite the right hand side of \eqref{eq:C-B} as \(\sum_{\mu\in\Par_n} (\det P^{[\mu]})(\det Q_{[\mu]})\), where
\( P^{[\mu]}=(p_{i,\mu_{n+1-j}+j-1})_{i,j \in [n]}\) and \(Q_{[\mu]}=(q_{\mu_{n+1-i}+i-1,j})_{i,j\in [n]}\).
Since \(\det P^{[\mu]} = (-1)^{n(n-1)/2} \det P^\mu\) and \(\det Q_{[\mu]} = (-1)^{n(n-1)/2} \det Q_\mu\),
   we obtain the result.
\end{proof}

For \( \lambda,\mu\in\Par_n \), define
\begin{align}
  \label{eq:C_lm} C_{\lambda,\mu}(\va,\vb) &= \det \left( h_{\mu_i-\lambda_j-i+j}[A_{\lambda_j}-B_{j-1}] \right)_{i,j=1}^{n},  \\
  \label{eq:c_lm} c_{\lambda,\mu}(\va,\vb) &= \det \left( h_{\lambda_i-\mu_j-i+j}[-A_{\lambda_i-1}+B_{i-1}] \right)_{i,j=1}^{n}.
\end{align}
Note that by Lemmas~\ref{lem:det(h)=0} and \ref{lem:det(h)=det(h)}, 
\( C_{\lambda,\mu}(\va,\vb) \) and 
\( c_{\lambda,\mu}(\va,\vb) \) are independent of the choice of \( n \)
provided \( \lambda,\mu\in\Par_n \). In particular, we can write
\begin{align*}
  C_{\lambda,\mu}(\va,\vb)
  &= \det \left( h_{\mu_i-\lambda_j-i+j}[A_{\lambda_j}-B_{j-1}] \right)_{i,j=1}^{\max(\ell(\lambda),\ell(\mu))},  \\
  c_{\lambda,\mu}(\va,\vb)
  &= \det \left( h_{\lambda_i-\mu_j-i+j}[-A_{\lambda_i-1}+B_{i-1}] \right)_{i,j=1}^{\max(\ell(\lambda),\ell(\mu))}.
\end{align*}
Note also that, by Lemma~\ref{lem:det(h)=0}, we have
\begin{subequations}
\begin{align}
  \label{eq:C_lm=0}
  C_{\lambda,\mu}(\va,\vb) &= 0 \quad
    \mbox{unless \( \mu\supseteq\lambda \), and,}  \\
  \label{eq:c_lm=0}
  c_{\lambda,\mu}(\va,\vb) &= 0 \quad
    \mbox{unless \( \mu\subseteq\lambda \).}
\end{align}
\end{subequations}
Thus, for a fixed \( \lambda\in\Par \), and any family
\( \{f_\mu\}_{\mu\in \Par} \) of elements in \( \hat{\Lambda} \), we have
\begin{align}
  \label{eq:sum_C}
  \sum_{\mu\in \Par}C_{\lambda,\mu}(\va,\vb) f_\mu
  &=\sum_{\mu \supseteq \lambda}C_{\lambda,\mu}(\va,\vb) f_\mu,\\
  \label{eq:sum_c}
  \sum_{\mu\in \Par}c_{\lambda,\mu}(\va,\vb) f_\mu
  &=\sum_{\mu \subseteq \lambda}c_{\lambda,\mu}(\va,\vb) f_\mu.
\end{align}

Now we expand \( G_{\lambda}(\vx_n;\va,\vb) \) and \(
g_{\lambda}(\vx_n;\va,\vb) \) into Schur functions.

\begin{thm}
  \label{thm:Schur_exp G,g for Par_n}
  Let \(\lambda \in \Par_n\). We have
  \begin{align}
    \label{eq:G_n=s*det} G_{\lambda}(\vx_n;\va,\vb) &= \sum_{\mu \supseteq \lambda}C_{\lambda,\mu}(\va,\vb) s_{\mu}(\vx_n),  \\
    \label{eq:g_n=s*det} g_{\lambda}(\vx_n;\va,\vb) &= \sum_{\mu \subseteq \lambda}c_{\lambda,\mu}(\va,\vb) s_{\mu}(\vx_n).
  \end{align}
\end{thm}
\begin{proof}
We first note that in \eqref{eq:G_n=s*det} the summation index \( \mu \) can be any
partition \( \mu\in \Par_m \), \( m\ge0 \), containing \( \lambda \). However,
since \( s_\mu(\vx_n)=0 \) if \( \ell(\mu)> n \), we will show
\eqref{eq:G_n=s*det} where the sum is over all \( \mu\in\Par_n \)
with \( \mu\supseteq\lambda \).

By transposing the matrix in \eqref{eq:defG'}, we have
\[
  G_\lambda(\vx_n;\va,\vb)
  = \frac{\det
    \left( h_{\lambda_j+n-j}[x_i \ominus(A_{\lambda_j}-B_{j-1})]
    \right)_{1\le i,j\le n}}
  {\prod_{1\le i<j\le n}(x_i-x_j)}.
\]
Note that we can rewrite the \( (i,j) \)-entry \( h_{\lambda_j+n-j}[x_i \ominus(A_{\lambda_j}-B_{j-1})] \) as
\[
  \sum_{k\ge0} h_k[x_i] h_{k-(\lambda_j+n-j)}[A_{\lambda_j}-B_{j-1}]
  = \sum_{k\ge0} x_i^k h_{k-(\lambda_j+n-j)}[A_{\lambda_j}-B_{j-1}],
\]
which is the \( (i,j) \)-entry of the matrix \( PQ \),
where 
\[
  P=(x_i^j)_{i\in [n],j\in\NN}, \qquad 
  Q=(h_{i-(\lambda_j+n-j)}[A_{\lambda_j}-B_{j-1}])_{i\in \NN, j\in[n]}.
\]
Thus
\[
  G_\lambda(\vx_n;\va,\vb)  = \frac{\det(PQ)}{\prod_{1\le i<j\le n}(x_i-x_j)}
  = \sum_{\mu \in \Par_n} \frac{\det P^\mu  \det Q_{\mu}}{\prod_{1\le i<j\le n}(x_i-x_j)},
\]
where Lemma~\ref{lem:C-B} is used for the second equality.
Since 
\[
  \frac{\det P^\mu}{\prod_{1\le i<j\le n}(x_i-x_j)}
  =\frac{\det (x_i^{\mu_j+n-j})_{i,j\in[n]}}{\prod_{1\le i<j\le n}(x_i-x_j)} = s_\mu(\vx_n)
\]
and
\( \det Q_{\mu} = C_{\lambda,\mu}(\va,\vb) \), we obtain
\[
  G_{\lambda}(\vx_n;\va,\vb) = \sum_{\mu\in\Par_n}C_{\lambda,\mu}(\va,\vb) s_{\mu}(\vx_n).
\]
Therefore we obtain the first identity \eqref{eq:G_n=s*det} from the above
identity and \eqref{eq:sum_C}.

The second identity \eqref{eq:g_n=s*det} can be proved similarly using
\[
  g_\lambda(\vx_n;\va,\vb)
  = \frac{\det(h_{\lambda_j+n-j}[x_i-A_{\lambda_j-1}+B_{j-1}])_{1\le i,j\le n}}
  {\prod_{1\le i<j\le n}(x_i-x_j)},
\]
\[
  h_{\lambda_j+n-j}[x_i-A_{\lambda_j-1}+B_{j-1}]
  = \sum_{k\ge0} x_i^k h_{\lambda_j+n-j-k}[-A_{\lambda_j-1}+B_{j-1}],
\]
and the matrices
\[
  P=(x_i^j)_{i\in [n],j\in\NN}, \qquad 
  R=(h_{\lambda_j+n-j-i}[-A_{\lambda_j-1}+B_{j-1}])_{i\in \NN, j\in[n]}. \qedhere
\]
\end{proof}

By taking the limit \( n\to\infty \) in Theorem~\ref{thm:Schur_exp G,g for Par_n} we obtain the following corollary.

 \begin{cor}
  \label{cor:Schur_exp G,g}
  For \(\lambda \in \Par\), we have
  \begin{align}
    \label{eq:Schur_exp G}G_{\lambda}(\vx;\va,\vb) &= \sum_{\mu \supseteq \lambda}C_{\lambda,\mu}(\va,\vb) s_{\mu}(\vx),  \\
    \label{eq:Schur_exp g}g_{\lambda}(\vx;\va,\vb) &= \sum_{\mu \subseteq \lambda}c_{\lambda,\mu}(\va,\vb) s_{\mu}(\vx).
  \end{align}
 \end{cor}

 We can apply the Cauchy--Binet theorem again to the Schur expansions in
 Theorem~\ref{thm:Schur_exp G,g for Par_n} to obtain Jacobi--Trudi-like formulas.

\begin{thm}\label{thm:JT_ab}
  For \( \lambda\in\Par_n \), we have
  \begin{align}
    \label{eq:JT_G}
    G_\lambda(\vx_n;\va,\vb)
    &= \det\left(  h_{\lambda_i-i+j}[X_n\ominus (A_{\lambda_i}-B_{i-1})]\right)_{i,j=1}^n,\\
    \label{eq:JT_g}
    g_\lambda(\vx_n;\va,\vb)
    &= \det\left(h_{\lambda_i-i+j}[X_n-A_{\lambda_i-1}+B_{i-1}]\right)_{i,j=1}^n.
  \end{align}
\end{thm}

\begin{proof}
  Let 
  \[
    P=(h_{j-(\lambda_i+n-i)}[A_{\lambda_i}-B_{i-1}])_{i\in[n],j\in \NN},
    \qquad Q=(h_{i-(n-j)}[X_n])_{i\in\NN,j\in [n]}.
  \]
  Then
  \begin{align*}
    \det P^\mu
    &= \det \left( h_{\mu_j+n-j-(\lambda_i+n-i)}[A_{\lambda_i}-B_{i-1}] \right)_{i,j\in[n]}=C_{\lambda,\mu}(\va,\vb),\\
    \det Q_\mu
    &= \det(h_{\mu_i+n-i-(n-j)}[X_n])_{i,j\in[n]}= \det(h_{\mu_i-i+j}[X_n])_{i,j\in[n]} =s_\mu(\vx_n),
  \end{align*}
  where the last equality is essentially the classical Jacobi--Trudi identity \eqref{eq:JT}.
  Thus, by \eqref{eq:G_n=s*det} and Lemmas~\ref{lem:det(h)=0} and \ref{lem:C-B}, we have
  \[
    G_\lambda(\vx_n;\va,\vb) = \sum_{\mu\supseteq\lambda} \det P^\mu \det Q_\mu 
    = \sum_{\mu\in\Par_n} \det P^\mu \det Q_\mu = \det(PQ).
  \]
  On the other hand,
  \begin{align*}
    \det(PQ)
    &= \det \left( \sum_{k \geq 0} h_{k-(\lambda_i+n-i)}[A_{\lambda_i} - B_{i-1}] h_{k-(n-j)}[X_n]  \right)_{i,j=1}^n\\
    &= \det \left( \sum_{a-b=\lambda_i-i+j} h_{a}[X_n] h_b[A_{\lambda_i} - B_{i-1}] \right)_{i,j=1}^n \\
    &=\det\left(  h_{\lambda_i-i+j}[X_n\ominus (A_{\lambda_i}-B_{i-1})]\right)_{i,j=1}^n,
  \end{align*}
  where the second equality is obtained by first considering the sum over all \(
  k\in\ZZ \) and then using the fact \( h_m(\vx)=0 \) for \( m<0 \).
  This shows the first identity \eqref{eq:JT_G}.

  The second identity \eqref{eq:JT_g} can be proved similarly by observing
  \( c_{\lambda,\mu}(\va,\vb) = \det R^\mu \) and \(s_\mu(\vx_n)=\det Q_\mu \), where
  \[
    R=(h_{\lambda_i +n -i -j }[- A_{\lambda_i -1} + B_{i-1}])_{i\in [n],j\in
      \NN},
    \qquad  Q=(h_{i-(n-j)}[X_n])_{i\in\NN,j\in [n]}. \qedhere
  \]
\end{proof}

Finally, we show the duality between
\( G_\lambda(\vx;\va,\vb) \) and \( g_\lambda(\vx;\va,\vb) \)
with respect to the Hall inner product
using their Schur expansions.

\begin{thm}\label{thm:vec dual}
  For \( \lambda,\mu\in\Par \), we have
  \[
    \langle G_\lambda(\vx;\va,\vb), g_\mu(\vx;\va,\vb)  \rangle
  = \delta_{\lambda,\mu}.
\]
\end{thm}
\begin{proof}
  By Corollary~\ref{cor:Schur_exp G,g}, we obtain
  \begin{align}
    \label{eq:G,g}
    \langle G_\lambda(\vx;\va,\vb), g_\mu(\vx;\va,\vb) \rangle
    &= \left\langle \sum_{\nu\supseteq\lambda} C_{\lambda,\nu}(\va,\vb) s_\nu(\vx),
      \sum_{\rho\subseteq\mu} c_{\mu,\rho}(\va,\vb) s_\rho(\vx) \right\rangle \\
    \notag
    &= \sum_{\lambda\subseteq\nu\subseteq\mu}
    C_{\lambda,\nu}(\va,\vb) c_{\mu,\nu}(\va,\vb).
  \end{align}
  Thus, if \( \lambda\not\subseteq\mu \), then \( \langle
  G_\lambda(\vx;\va,\vb), g_\mu(\vx;\va,\vb) \rangle=0 \)
  and the theorem holds. Therefore we may assume \( \lambda\subseteq\mu \).

  Fix an integer \( n \) such that \( \lambda,\mu\in\Par_n \) and let
  \[
    P=(h_{\mu_i+n-i-j}[-A_{\mu_i-1}+B_{i-1}])_{i\in [n], j\in\NN}, \quad
    Q=(h_{i-(\lambda_j+n-j)}[A_{\lambda_j}-B_{j-1}])_{i\in\NN,j\in[n]}.
  \]
  Then
  \begin{align*}
    c_{\mu,\nu}(\va,\vb)
    &=\det \left( h_{\mu_i-\nu_j-i+j}[-A_{\mu_i-1}+B_{i-1}] \right)_{i,j=1}^n = \det P^\nu,\\
    C_{\lambda,\nu}(\va,\vb)
    &=\det \left( h_{\nu_i-\lambda_j-i+j}[A_{\lambda_j}-B_{j-1}] \right)_{i,j=1}^n= \det Q_\nu.
  \end{align*}
  Thus, by \eqref{eq:G,g} and Lemmas~\ref{lem:det(h)=0} and \ref{lem:C-B}, we have 
  \[
    \langle G_\lambda(\vx;\va,\vb), g_\mu(\vx;\va,\vb) \rangle
    = \sum_{\lambda\subseteq\nu\subseteq\mu} \det P^\nu \det Q_\nu
    = \sum_{\nu\in \Par_n} \det P^\nu \det Q_\nu = \det(PQ).
  \]
  On the other hand, the \( (i,j) \)-entry of \( PQ \) is
  \[
    \sum_{k\ge0} h_{\mu_i+n-i-k}[-A_{\mu_i-1}+B_{i-1}]h_{k-(\lambda_j+n-j)}[A_{\lambda_j}-B_{j-1}]
    =h_{\mu_i-\lambda_j-i+j}[A_{\lambda_j}-B_{j-1}-A_{\mu_i-1}+B_{i-1}].
  \]
  Therefore we obtain
  \begin{equation}
    \label{eq:duality_proof}
    \langle G_\lambda(\vx;\va,\vb), g_\mu(\vx;\va,\vb) \rangle
    = \det \left( h_{\mu_i-\lambda_j-i+j}[A_{\lambda_j}-B_{j-1}-A_{\mu_i-1}+B_{i-1}] \right)_{i,j=1}^n .
  \end{equation}
It suffices to show that if \( \lambda\subseteq\mu \), then
\begin{equation}
  \label{eq:det=d}
   \det \left( h_{\mu_i-\lambda_j-i+j}[A_{\lambda_j}-B_{j-1}-A_{\mu_i-1}+B_{i-1}] \right)_{i,j=1}^n
  = \delta_{\lambda,\mu}.
\end{equation}

We proceed by induction on \( n \). It is clear when \( n=0 \). Let \( n\ge 1\)
and suppose \eqref{eq:det=d} holds for all integers less than \( n \). If there
is an integer \( 1\le k\le n \) with \( \mu_k=\lambda_k \), then by
Lemma~\ref{lem:det(h)=det*det}, we can reduce the size of \( n \) and use the
induction hypothesis (after shifting the indices of \( \beta_i \) in the second
matrix in the right hand side of the lemma by \( k \)). Therefore we may assume \( \mu_k\ge \lambda_k+1 \) for all
\( 1\le k\le n \).

In this case it suffices to prove this claim: the matrix in \eqref{eq:det=d} is a lower
triangular matrix whose diagonal entries are all zero. Let \( 1 \le i \le j \le
n \). Then \( \mu_i\ge \mu_j\ge \lambda_j+1 \) and therefore 
\( \mu_i-\lambda_j-i+j>0 \) and the \( (i,j) \)-entry in \eqref{eq:det=d} is equal to
  \begin{align*}
    &h_{\mu_i-\lambda_j-i+j}
      [-\alpha_{\lambda_j+1}-\dots-\alpha_{\mu_i-1}-\beta_i-\dots-\beta_{j-1}] \\
    &\qquad\quad = (-1)^{\mu_i-\lambda_j-i+j} e_{\mu_i-\lambda_j-i+j}
      (\alpha_{\lambda_j+1},\dots,\alpha_{\mu_i-1},\beta_i,\dots,\beta_{j-1}).
  \end{align*}
  Since \( \mu_i-\lambda_j-i+j \) is greater than the number of variables in the
  above elementary symmetric function, the \( (i,j) \)-entry is equal to \( 0
  \). This shows the claim and the proof follows by induction.
\end{proof}

By Theorem~\ref{thm:vec dual} and \eqref{eq:G,g}, and
\eqref{eq:C_lm=0} and \eqref{eq:c_lm=0}, for any
\( \lambda,\mu\in\Par \),
\begin{equation}\label{eq:Cc=d1}
  \sum_{\nu} C_{\lambda,\nu}(\va,\vb) c_{\mu,\nu}(\va,\vb) = \delta_{\lambda,\mu}.
\end{equation}
We also claim that the two matrices \(
(C_{\lambda,\mu}(\va,\vb))_{\lambda,\mu\in\Par} \) and \(
(c_{\mu,\lambda}(\va,\vb))_{\lambda,\mu\in\Par} \) are inverses of each other.
To derive this from \eqref{eq:Cc=d1}, it is enough to show that the matrix
\( (C_{\lambda,\mu}(\va,\vb))_{\lambda,\mu\in\Par} \) is invertible.
Due to \eqref{eq:C_lm=0} and the fact that \( C_{\lambda,\lambda}(\va,\vb)=1 \)
for all \( \lambda\in\Par \) from the definition \eqref{eq:C_lm}, the matrix is
triangular with ones on the diagonal.
Hence, \( I-(C_{\lambda,\mu}(\va,\vb))_{\lambda,\mu\in\Par} \) is
locally nilpotent, and \( \sum_{k\ge 0} (I-(C_{\lambda,\mu}(\va,\vb))_{\lambda,\mu\in\Par})^k  \) serves as an inverse of the matrix \( (C_{\lambda,\mu}(\va,\vb))_{\lambda,\mu\in\Par} \) where \( I=(\delta_{\lambda,\mu})_{\lambda,\mu\in\Par} \).
The claim gives us
\begin{equation}\label{eq:Cc=d2}
  \sum_{\nu} c_{\nu,\lambda}(\va,\vb) C_{\nu,\mu}(\va,\vb) = \delta_{\lambda,\mu}.
\end{equation}

As consequences of Theorem~\ref{thm:vec dual}, we have the following
corollaries: a Cauchy-type identity and Grothendieck expansions of a Schur function.

\begin{cor}
  We have
  \[
    \prod_{i,j\ge 1} \frac{1}{1-x_i y_j} = \sum_{\lambda\in\Par} G_\lambda(\vx;\va,\vb) g_\lambda(\vy;\va,\vb),
  \]
  where \( \vx=(x_1,x_2,\dots) \) and \( \vy=(y_1,y_2,\dots) \).
\end{cor}
\begin{proof}
  We proceed similarly as in the proof of \cite[p.~63, (4.6)]{Macdonald}.
The right hand side is 
  \[
    \sum_{\lambda} \sum_{\mu} C_{\lambda,\mu}(\va,\vb)
    s_\mu(\vx) \sum_{\nu} c_{\lambda,\nu}(\va,\vb) s_\nu(\vy) =
    \sum_{\mu,\nu} \left( \sum_{\lambda}
      C_{\lambda,\mu}(\va,\vb)c_{\lambda,\nu}(\va,\vb)\right) s_\mu(\vx)s_\nu(\vy).
  \]
  By \eqref{eq:Cc=d2},  the above sum is equal to \( \sum_{\mu}s_\mu(\vx)s_\mu(\vy) \). The
  Cauchy identity \( \sum_{\mu}s_\mu(\vx)s_\mu(\vy)=\prod_{i,j\ge 1} (1-x_i y_j)^{-1} \)
  completes the proof.
\end{proof}

Since \( (C_{\lambda,\mu}(\va,\vb))_{\lambda,\mu\in\Par}\) is the inverse of
\((c_{\mu,\lambda}(\va,\vb))_{\lambda,\mu\in\Par} \),
Corollary~\ref{cor:Schur_exp G,g} implies the following result.

\begin{cor}
  For \( \lambda\in\Par \), we have
  \begin{align}
    \label{eq:G_expansion} s_\lambda(\vx)&=\sum_{\mu\supseteq \lambda} c_{\mu,\lambda}(\va,\vb) G_\mu(\vx;\va,\vb) ,\\
    \label{eq:g_expansion}s_\lambda(\vx)&=\sum_{\mu\subseteq \lambda} C_{\mu,\lambda}(\va,\vb) g_\mu(\vx;\va,\vb).
  \end{align}
\end{cor}

By Theorem~\ref{thm:vec dual}, both \(
\{G_\lambda(\vx;\va,\vb)\}_{\lambda\in\Par} \) and \(
\{g_\lambda(\vx;\va,\vb)\}_{\lambda\in\Par} \) are linearly independent sets.
Thus, \eqref{eq:g_expansion} implies that \(
\{g_\lambda(\vx;\va,\vb)\}_{\lambda\in\Par} \) is a basis of \(
\Lambda_{\QQ[\va,\vb]} \). Note that, even though \(
\{G_\lambda(\vx;\va,\vb)\}_{\lambda\in\Par}\not\subseteq \Lambda_{\QQ[\va,\vb]}
\), this set also behaves like a basis of \( \Lambda_{\QQ[\va,\vb]} \) in the
sense that every element of \( \Lambda_{\QQ[\va,\vb]} \) can be written
uniquely as an infinite linear combination of \( G_\lambda(\vx;\va,\vb) \)'s.

\section{Combinatorial models and Schur positivity}
\label{sec:lattice}

In this section we give combinatorial models for the Schur coefficients
\(C_{\lambda,\mu}(\va,\vb) \) and \(c_{\lambda,\mu}(\va,\vb)\) of
\(G_{\lambda}(\vx;\va,\vb)\) and \(g_{\lambda}(\vx;\va,\vb)\) obtained in the
previous section. As a consequence we show that \(G_{\lambda}(\vx;\va,-\vb)\)
and \(g_{\lambda}(\vx;-\va,\vb)\) are Schur-positive. This generalizes a result of
Yeliussizov~\cite[Section~9]{Yeliussizov2017} who found a lattice path
model and a tableau model for the Schur coefficients of \(
g_\lambda^{(\alpha,\beta)}(\vx)=g_{\lambda}(\vx;(-\alpha,-\alpha,\dots),
(\beta,\beta,\dots)) \) and consequently obtained the Schur positivity of
\( g_\lambda^{(\alpha,\beta)}(\vx) \).
Using crystals, Hawkes and Scrimshaw~\cite[Section~4]{HS20} showed the Schur
positivity of \(G^{(\alpha,\beta)}_{\lambda}(\vx) =
G_{\lambda}(\vx;(\alpha,\alpha,\dots),(-\beta,-\beta,\dots))\).

In order to state our results we need the following tableaux, which are modified
versions of elegant tableaux \cite{LP2007,Lenart2000}.

\begin{defn}
  A \emph{\( \ZZ \)-elegant tableau} of shape $\lm$ is a filling $T$ of the
  cells in $\lm$ with integers such that the rows are weakly increasing, the
  columns are strictly increasing, and
  \begin{equation}
    \label{eq:ETc1}
    \min(i-\mu_i,1) \le T(i,j) < i \qquad\mbox{for all $(i,j)\in \lm$,}
  \end{equation}
  or equivalently,
  \begin{equation}
    \label{eq:ETc2}
    \min(i-j,0) < T(i,j) < i \qquad\mbox{for all $(i,j)\in \lm$.}
  \end{equation}
Let $\ET_{\ZZ}(\lm)$ denote the set of all $\ZZ$-elegant
  tableaux of shape $\lm$. See Figure~\ref{fig:ET and IET}.
\end{defn}

The equivalence of \eqref{eq:ETc1} and \eqref{eq:ETc2} can be checked easily
using the condition that the rows are weakly increasing.

\begin{defn}
  A \emph{\( \ZZ \)-inelegant tableau} of shape $\ml$ is a filling $T$ of the
  cells in $\ml$ with integers such that the rows are weakly decreasing, the
  columns are strictly decreasing, and
  \begin{equation}
    \label{eq:IETc1}
    \min(\mu_i-i,0) < T(i,j) \le \lambda_i \qquad\mbox{for all $(i,j)\in \ml$,}
  \end{equation}
  or equivalently,
  \begin{equation}
    \label{eq:IETc2}
    \min(j-i,0) < T(i,j) < j \qquad\mbox{for all $(i,j)\in \ml$.}
  \end{equation}
  Let $\IET_{\ZZ}(\ml)$ denote the set of all $\ZZ$-inelegant tableaux of shape
  $\ml$. See Figure~\ref{fig:ET and IET}.
\end{defn}

The equivalence of \eqref{eq:IETc1} and \eqref{eq:IETc2} can be checked easily
using the condition that the rows are weakly decreasing.

\begin{figure}
  \centering
  \begin{ytableau}
    \none[-2 \le] & \none & \none & \none & \none& -1 &  \none &\none[<1]\\
    \none[\phantom{-} 0 \le] & \none & \none & \none & 0 & 1  &\none &\none[<2]\\
    \none[\phantom{-} 1 \le] & \none & 1 & 2 & 2 & 2 &\none &\none[<3]\\
    \none[\phantom{-} 1 \le] &\none & 2 & 3 & \none & \none & \none &\none[<4]

  \end{ytableau}\qquad\qquad\qquad
  \begin{ytableau}
    \none[3\ge] &\none& \none & \none & \none & 3 & 3 & 2 &\none & \none[>0\phantom{-}]\\
    \none[3\ge] &\none&\none & \none & \none & 2 & 2 & 1 &\none& \none[>0\phantom{-}]\\
    \none[1\ge] &\none&\none & 1 & 1 & 1 &\none&\none&\none & \none[>0\phantom{-}]\\
    \none[0\ge] &\none&0 & -1  & \none &\none&\none&\none&\none& \none[>-2]
  \end{ytableau}
  \caption{A \( \ZZ \)-elegant tableau of shape \( (4,4,4,2)/(3,2) \) on
    the left and a \( \ZZ \)-inelegant tableau of shape \(
    (6,6,4,2)/(3,3,1) \) on the right.
    }
  \label{fig:ET and IET}
\end{figure}

We now give combinatorial interpretations for the Schur coefficients
\(C_{\lambda,\mu}(\va,\vb)\) and \(c_{\lambda,\mu}(\va,\vb)\) of
\(G_{\lambda}(\vx;\va,\vb)\) and \(g_{\lambda}(\vx;\va,\vb)\) defined in
\eqref{eq:C_lm} and \eqref{eq:c_lm}.

\begin{thm}\label{thm:C_lm}
  Let \(\lambda,\mu \in\Par\) with \(\lambda\subseteq\mu\).
  We have 
  \[
    C_{\lambda,\mu}(\va,\vb) = \sum_{T\in\IET_\ZZ(\ml)}
    \prod_{(i,j)\in\ml} (\alpha_{T(i,j)} - \beta_{T(i,j)-c(i,j)}),
  \]  
  where \( \alpha_m=\beta_m=0 \) for \( m\le 0 \). 
  In particular, $G_\lambda(\vx;\va,-\vb)$ is Schur-positive, that is,
  $C_{\lambda,\mu}(\va,-\vb)$ is a polynomial in $\va$ and $\vb$ with
  nonnegative integer coefficients. 
\end{thm}

\begin{thm}\label{thm:c_lm}
  Let \(\lambda,\mu \in\Par\) with \(\mu\subseteq\lambda\).
  We have
  \[
    c_{\lambda,\mu}(\va,\vb) = \sum_{T\in\ET_\ZZ(\lm)}
    \prod_{(i,j)\in\lm} (-\alpha_{T(i,j)+c(i,j)} + \beta_{T(i,j)}),
  \]  
  where \( \alpha_m=\beta_m=0 \) for \( m\le 0 \).
  In particular, $g_\lambda(\vx;-\va,\vb)$ is Schur-positive, that is,
  $c_{\lambda,\mu}(-\va,\vb)$ is a polynomial in $\va$ and $\vb$ with
  nonnegative integer coefficients. 
\end{thm}

By Corollary~\ref{cor:Schur_exp G,g} and the above theorems, we obtain
combinatorial interpretations for \( G_\lambda(\vx;\va,\vb) \) and \(
g_\lambda(\vx;\va,\vb) \) using pairs of tableaux as follows.

\begin{cor}\label{cor:G_comb2}
  For \( \lambda\in \Par \), we have
  \begin{align*}
    G_\lambda(\vx;\va,\vb) = \sum_{\mu\supseteq\lambda}
    \sum_{(E,T)\in \IET_{\ZZ}(\mu/\lambda) \times \SSYT(\mu)} \wt(E) \vx^{T},
  \end{align*}
  where \( \wt(E) = \prod_{(i,j)\in \ml} (\alpha_{E(i,j)}-\beta_{E(i,j)-c(i,j)}) \).
\end{cor}

\begin{cor}\label{cor:g_comb2}
  For \( \lambda\in \Par \), we have
  \begin{align*}
    g_\lambda(\vx;\va,\vb) = \sum_{\mu\subseteq \lambda}\sum_{(E,T) \in \ET_{\ZZ}(\lm) \times \SSYT(\mu)} \wt'(E) \vx^{T},
  \end{align*}
  where \( \wt'(E) = \prod_{(i,j)\in \lm} (-\alpha_{E(i,j)+c(i,j)}+\beta_{E(i,j)}) \).
\end{cor}

Before proving Theorems~\ref{thm:C_lm} and \ref{thm:c_lm} we show that these
theorems generalize the combinatorial models for the Schur coefficients of \(
G_\lambda(\vx)= G_\lambda(\vx;\bm0,\bm1) \) and \(
g_\lambda(\vx)=g_\lambda(\vx;\bm0,\bm1) \) due to Lenart \cite{Lenart2000} and
Lam and Pylyavskyy \cite{LP2007}, where \( \bm0=(0,0,\dots) \) and \(
\bm1=(1,1,\dots) \). Note that Lenart \cite{Lenart2000} uses a different
definition for the Grothendieck polynomial and denotes it by \(
\mathfrak{G}_\lambda \). It follows from \cite[(2.2),(2.3)]{Lenart2000} that \(
\mathfrak{G}_\lambda = G_\lambda(\vx) \).

Lenart \cite[Theorem~2.2]{Lenart2000} showed that
\[
  G_\lambda(\vx)= \sum_{\mu\supseteq\lambda} (-1)^{|\ml|} C_{\lambda,\mu} s_\mu(\vx),
\]
where \( C_{\lambda,\mu} \) is the number of tableaux \( T \) of shape \( \ml \)
with integer entries such that the rows and columns are strictly increasing, and
\(0< T(i,j) <i \) for all \( (i,j)\in\ml \). For such a tableau \( T \), define
\( R \) to be the tableau of shape \( \ml \) such that \( R(i,j) = j-T(i,j) \)
for all \( (i,j)\in\ml \). Then the tableau \( R \) satisfies the conditions
that the rows are weakly decreasing, the columns are strictly decreasing, and \(
j-i<R(i,j)<j \), and \( C_{\lambda,\mu} \) is equal to the number of such
tableaux. 

On the other hand, by Theorem~\ref{thm:C_lm}, we have \(
C_{\lambda,\mu}(\bm0,\bm1)=(-1)^{|\ml|}C'_{\lambda,\mu} \), where
\( C'_{\lambda,\mu} \) is the number of \( S\in\IET_\ZZ(\ml) \) such that
\( S(i,j)-c(i,j)>0 \) for all \( (i,j)\in\ml \). This condition combined with
\eqref{eq:IETc2} can be written as \( j-i<S(i,j)<j \) for all \( (i,j)\in\ml \).
Thus \( C'_{\lambda,\mu}=C_{\lambda,\mu} \).

Lam and Pylyavskyy \cite[Theorem~9.8]{LP2007} showed that 
\[
  g_\lambda(\vx) = \sum_{\mu\subseteq\lambda}c_{\lambda,\mu} s_\mu(\vx),
\]
where \( c_{\lambda,\mu} \) is the number of semistandard Young tableaux \( T \)
of shape \( \lm \) such that \( T(i,j)<i \) for \( (i,j)\in\lm \).
It follows immediately from Theorem~\ref{thm:c_lm} that
\( c_{\lambda,\mu}(\bm{0},\bm{1})=c_{\lambda,\mu} \).

\medskip

For the rest of this section we prove Theorems~\ref{thm:C_lm} and
\ref{thm:c_lm}. We first express \(C_{\lambda,\mu}(\va,\vb) \) and
\(c_{\lambda,\mu}(\va,\vb)\) in terms of nonintersecting lattice paths using the
Lindstr\"om--Gessel--Viennot lemma \cite{LGV,Lindstrom}, see also
\cite[Theorem~2.7.1]{EC1}. The nonintersecting lattice paths are then easily
transformed into the desired tableaux. We begin with some definitions.

A \emph{weighted lattice graph} is a directed graph with vertex set \(
\ZZ\times\ZZ \) in which every edge is assigned a weight. A directed edge from
\( (a,b) \) to \( (c,d) \) is denoted by \( (a,b)\to(c,d) \) and its
weight is denoted by \( \wt((a,b)\to(c,d)) \).
In this paper, we only consider three kinds of edges, namely, an \emph{east
  step} \( (a,b)\to(a+1,b) \), a \emph{north step} \( (a,b)\to(a,b+1) \), and a
\emph{west step} \( (a,b)\to(a-1,b) \).

Let \( \Gamma \) be a weighted lattice graph.  For two vertices \( u \) and \( v \) of \( \Gamma
\), a \emph{path} from \( u \) to \( v \) is a sequence of vertices \(
(u_0,u_1,\dots,u_k) \) of \( \Gamma \) such that \( u_0=u \), \( u_k=v \) and each
step \( u_i\to u_{i+1} \) is a directed edge in \( \Gamma \). For a path \( p \),
define \( \wt(p) \) to be the product of the weights of the edges in \( p \).
For vertices \( u \) and \( v \), denote by \( P_{\Gamma}(u,v)\) the set of paths from \(
u \) to \( v \) and define
\[
  \wt(P_{\Gamma}(u,v))=\sum_{p\in P_{\Gamma}(u,v)}{\wt(p)}.
\]
An \emph{\( n \)-path} is an \( n \)-tuple \( \vp=(p_1,\dots,p_n) \) of
paths. We define
\[
  \wt(\vp)=\prod^n_{i=1}{\wt(p_i)}.
\]

For the rest of this section we assume \( \alpha_k=\beta_k=0 \) for \( k\le0 \).

The following two propositions give combinatorial meanings of the plethystic
substitution of the form \( h_r[A_s-B_t] \) using lattice paths.

\begin{prop} \label{prop: hm[A-B]}
  Let \( \Gamma \) be the weighted lattice graph with edge set
  \[
    \{(a,b)\to (a-1,b):a,b\in\ZZ\} \cup \{(a,b)\to (a,b+1):a,b\in\ZZ\},
  \]
  where the weights on edges are given by
  \( \wt((a,b)\to(a-1,b))=\alpha_b-\beta_{-a+b} \) and
  \( \wt((a,b)\to(a,b+1))=1 \).
  For integers \(r,s\), and \(t\) with \( s\ge \min(t,0) \), we have
  \begin{equation*}
    h_{t-r}[A_s-B_{s-r-1}] = \wt(P_{\Gamma}((t,\min(t,0)),(r,s))),
  \end{equation*}
  where \( A_i = B_i = 0 \) for \( i\le 0 \) by convention.
\end{prop}
\begin{proof}
  Let \( L(r,s,t) \) (resp.~\( R(r,s,t) \)) be the left hand side (resp.~right hand side) of the equation. We will show that 
  \( L(r,s,t)=R(r,s,t) \) for all integers \( r, s \), and \( t \) with
  \( s\ge\min(t,0) \) by induction on the pair \( (r,s) \) when \( t \)
  is fixed. We will only show the case \( t\ge0 \) because the other case
  \( t<0 \) can be proved similarly.
  
  Let \( t \) be a fixed integer with \( t\ge0 \). We proceed by induction
  on \( (r,s) \) with base cases being \( r>t \) or \( s=0 \). Note that
  \( R(r,s,t) \) is the sum of the weights of the paths from \( (t,0) \)
  to \( (r,s) \) and the west step \( (t,0)\to(t-1,0) \) from the starting point
  has weight \( 0\).
  If \( r>t \), then by definition we have
  \( L(r,s,t)=R(r,s,t)=0 \). If \( r\le t \) and \( s=0 \), then
  \begin{align*}
    L(r,s,t) &=h_{t-r}[-B_{-r-1}]=(-1)^{t-r}e_{t-r}[B_{-r-1}]=\delta_{r,t},\\
    R(r,s,t) &=\wt(P_{\Gamma}((t,0),(r,0)))
    =\delta_{r,t},
  \end{align*}
  and therefore \( L(r,s,t)=R(r,s,t) \).
  Here, we used the fact that \( e_a[B_b]=e_a(\beta_1,\dots,\beta_b)=0 \)
  if \( 0<a \) and \( b < a \).

  Now let \( M\le t \) and \( N>0 \), and suppose that the result holds
  when \( (r,s) \) is \( (M,N-1) \) or \( (M+1,N) \). Consider the case
  \( (r,s)=(M,N) \). By applying Lemma~\ref{lem:h_m[Z]} twice, we have
  \begin{align*}
    & L(r,s,t)=h_{t-r}[A_s-B_{s-r-1}] =h_{t-r}[A_s-B_{s-r-2}]-\beta_{s-r-1}h_{t-r-1}[A_s-B_{s-r-2}]\\
    &=(h_{t-r}[A_{s-1}-B_{s-r-2}]+\alpha_s h_{t-r-1}[A_s-B_{s-r-2}])-\beta_{s-r-1}h_{t-r-1}[A_s-B_{s-r-2}]\\
    &=L(r,s-1,t)+(\alpha_s-\beta_{s-r-1})L(r+1,s,t).
  \end{align*}
  On the other hand, by the definition
  of \( R(r,s,t) \), it is easy to see that
  \[
    R(r,s,t)=R(r,s-1,t)+(\alpha_s-\beta_{s-r-1})R(r+1,s,t).
  \]
  By the above two equations and the induction hypothesis, we obtain
  \( L(r,s,t)=R(r,s,t) \), and the proof follows by induction.
\end{proof}

\begin{prop} \label{prop: hm[-A+B]}
  Let \( \Gamma \) be the weighted lattice graph with edge set
  \[
    \{(a,b)\to (a+1,b):a,b\in\ZZ\} \cup \{(a,b)\to (a,b+1):a,b\in\ZZ\},
  \]
  where the weights on edges are given by
  \( \wt((a,b)\to(a+1,b))=-\alpha_{a+b+1}+\beta_b \) and
  \( \wt((a,b)\to(a,b+1))=1 \). For integers \(r,s\), and \(t\) with
  \( s\ge \min(-t-1,0) \), we have
  \begin{equation*}
    h_{r-t}[-A_{r+s}+B_s] = \wt(P_{\Gamma}((t,\min(-t-1,0)),(r,s))). 
  \end{equation*}
 \end{prop}

\begin{proof}
  This can be proved similarly as Proposition~\ref{prop: hm[A-B]}. Alternatively,
  this can be obtained from Proposition~\ref{prop: hm[A-B]} by applying the
  transformation \( (x,y)\mapsto (-x-1,y) \),
  replacing \( r \) and \( t \) by \( -r-1 \) and \( -t-1 \), respectively,
  and interchanging \( \va \) and \( \vb \).
\end{proof}

Using the above propositions we can interpret the \( (i,j) \)-entry of the
matrices in the
definitions of \( C_{\lambda,\mu}(\va,\vb) \) and \( c_{\lambda,\mu}(\va,\vb) \)
in \eqref{eq:C_lm} and \eqref{eq:c_lm} as the sums of weights of paths.

\begin{lem}\label{lem: h[A-B]}
  Let \( \Gamma \) be the weighted lattice graph in Proposition~\ref{prop: hm[A-B]}.
  For \(\lambda,\mu\in\Par_n\) and \(i,j\in \{1,\dots,n\}\), we have
  \begin{align}\label{eq:P(u+1,v)}
    h_{\mu_i-\lambda_j-i+j}[A_{\lambda_j}-B_{j-1}]
    =\wt(P_{\Gamma}((\mu_i-i,\min(\mu_i-i+1,1)),(\lambda_j-j,\lambda_j))).
  \end{align}
\end{lem}
\begin{proof}
  By Proposition~\ref{prop: hm[A-B]} with \(r=\lambda_j-j\),
  \(s=\lambda_j\), and \(t=\mu_i-i\), we have 
  \begin{align}\label{eq:P(u,v)}
    h_{\mu_i-\lambda_j-i+j}[A_{\lambda_j}-B_{j-1}]
    =\wt(P_{\Gamma}((\mu_i-i,\min(\mu_i-i,0)),(\lambda_j-j,\lambda_j))).
  \end{align}
  Since \( \lambda_j\ge0 \), we have \( \lambda_j \ge\min(\mu_i-i,0)  \).
  We consider the two cases \( \lambda_j >\min(\mu_i-i,0) \) and \( \lambda_j = \min(\mu_i-i,0) \).

  Suppose \( \lambda_j >\min(\mu_i-i,0)  \). It is straightforward to check that
  the weight of the west step \((\mu_i-i,\min(\mu_i-i,0)) \to
  (\mu_i-i-1,\min(\mu_i-i,0)) \) is zero. Thus when we compute
  \( \wt(P_{\Gamma}((\mu_i-i,\min(\mu_i-i,0)),(\lambda_j-j,\lambda_j))) \)
  we only need to consider paths starting with a north step.
  This implies that we can replace the starting point by \( (\mu_i-i,\min(\mu_i-i+1,1)) \)
  and the result follows.
  
  Now suppose \( \lambda_j = \min(\mu_i-i,0) \). Then \( \lambda_j=0 \) and \( \mu_i-i \ge 0 \). 
  Thus
  \begin{align*}
    P_{\Gamma}((\mu_i-i,\min(\mu_i-i+1,1)),(\lambda_j-j,\lambda_j)) 
    &=P_{\Gamma}((\mu_i-i,1),(-j,0)) = \emptyset, \\
    P_{\Gamma}((\mu_i-i,\min(\mu_i-i,0)),(\lambda_j-j,\lambda_j))
    &=P_{\Gamma}((\mu_i-i,0),(-j,0)) = \{p\},
  \end{align*}
  where \( p \) is the path consisting of \( \mu_i-i+j \) west steps. Since \( p
  \) starts with a west step whose weight is \( 0 \) we have \( \wt(p)=0 \).
  Therefore the right hand sides of \eqref{eq:P(u+1,v)} and \eqref{eq:P(u,v)}
  are both equal to \( 0 \), and we get the result.
\end{proof}

\begin{lem} \label{lem: h[-A+B]}
  Let \( \Gamma \) be the weighted lattice graph in Proposition~\ref{prop: hm[-A+B]}.
  For \(\lambda,\mu\in\Par_n\) with \( \lambda_1 > 0 \) and
 for \(i,j\in \{1,\dots,n\}\), we have
  \[
    h_{\lambda_i-\mu_j-i+j}[-A_{\lambda_i-1}+B_{i-1}]
    =\wt(P_{\Gamma}((\mu_j-j,\min(j-\mu_j,1)),(\lambda_i-i,i-1))).
  \]
\end{lem}
\begin{proof}
  Similar to Lemma~\ref{lem: h[A-B]},
  the proof follows from Proposition~\ref{prop: hm[-A+B]} with \(r=\lambda_i-i\), \(s=i-1\), and \(t=\mu_j-j\). 
\end{proof}

We are now ready to prove Theorems~\ref{thm:C_lm} and \ref{thm:c_lm}. We will
only prove Theorem~\ref{thm:C_lm}, because Theorem~\ref{thm:c_lm} can be proved
similarly using Lemma~\ref{lem: h[-A+B]}.

\begin{proof}[Proof of Theorem~\ref{thm:C_lm}]
  Let \( \Gamma \) be the graph defined in Proposition~\ref{prop: hm[A-B]}
  and suppose \( \ell(\mu)=n \). For \( i= 1,\dots,n \), let 
  \[ 
     u_i = (\mu_i-i,\min(\mu_i-i+1,1)) \quad \mbox{and} \quad v_i = (\lambda_i-i,\lambda_i).
   \] 
  Note that each \( u_{i} \) (resp.~\( v_{i} \))
  is weakly south and strictly west of \( u_{i-1} \) (resp.~\( v_{i-1}\)).
  By \eqref{eq:C_lm}, Lemma~\ref{lem: h[A-B]}, and the
  Lindstr\"om--Gessel--Viennot lemma, we have
  \begin{equation} \label{eq:C=nonintersecting}
    C_{\lambda,\mu}(\va,\vb) =\det\left( \wt(P_\Gamma (u_i,v_j))\right)_{i,j=1}^n
    =\sum_{\vp\in X}\wt(\vp),
  \end{equation}
  where \( X \) is the set of all nonintersecting \( n \)-paths
  \(\vp=(p_1,\dots,p_n)\) with \(p_i\in P_\Gamma(u_i,v_i)\).

  We will construct a bijection \( \varphi:X\to \IET_\ZZ(\ml) \) by modifying
  the well-known bijection \cite[Theorem~7.16.1]{EC2} between nonintersecting
  lattice paths and (reverse) semistandard Young tableaux.

  First, we set \( \overline{u}_i = ( \mu_i-i,-\infty ) \) and \( \overline{v}_i
  = (\lambda_i-i,\infty) \) for \( i=1,\dots,n \). Here, in order to avoid
  dealing with infinite paths, one can think of \( -\infty \) and \( \infty \) as
  sufficiently small and large integers, respectively.
  Consider a nonintersecting \( n \)-path
  \( \overline{\vp}=(\overline{p}_1, \dots,\overline{p}_n) \), where \(
  \overline{p}_i\in P_{\Gamma}(\overline{u}_i,\overline{v}_i) \). Let \( T \) be
  the tableau of shape \( \ml \) such that the heights of the west steps in \(
  \overline{p}_i \) are the entries in the \( i \)th row of \( T \). More
  precisely, the \(y\)-coordinate of the \(j\)th west step in \(\overline{p}_i\)
  is the entry \( T(i,\mu_i+1-j) \) of the \( j \)th cell in the \( i \)th row
  of \( \ml \) from the right, see Figure~\ref{fig:paths=IET}. The map \(
  \overline{\vp}\mapsto T \) is essentially the same map in 
  \cite[Proof of Theorem~7.16.1]{EC2} with the lower bound 1 replaced by \( -\infty \)
  and the direction reversed. Thus \( \overline{\vp}\mapsto T \) is a bijection
  from the set of nonintersecting \( n \)-paths \(
  \overline{\vp}=(\overline{p}_1,\dots,\overline{p}_n) \), where \(
  \overline{p}_i\in P_{\Gamma}(\overline{u}_i,\overline{v}_i) \), to the set of
  fillings \( T \) of \( \ml \) with integers such that the entries in each row
  are weakly decreasing and the entries in each column are strictly decreasing.

  Now we restrict this map to the \( n \)-paths \( \overline{\vp} \) such that
  \( \overline{p}_i \) passes through \( u_i \) and \( v_i \) for all \(
  i=1,\dots,n \). Such an \( n \)-path \( \overline{\vp} \) can be identified
  with the \( n \)-path \( \vp=(p_1,\dots,p_n)\in X \), where \( p_i \) is the
  part of \( \overline{p}_i \) from \( u_i \) to \( v_i \).
  Moreover, \( \vp \) is nonintersecting if and only if \( \overline{\vp} \) is,
  because each \( u_i \) (resp.~\( v_i \)) is weakly south and strictly west of
  \( u_{i-1} \) (resp.~\( v_{i-1} \)). Since \( \overline{p}_i \) does not have a
  west step below \( u_i \) and above \( v_i \), if \( \overline{\vp}\mapsto T \),
  then \( \min(\mu_i-i,0)< T(i,j) \le \lambda_i \) for all \( (i,j)\in\ml \).
  Therefore the restriction of \( \overline{\vp}\mapsto T \) gives a desired
  bijection \( \varphi:X\to \IET_\ZZ(\ml) \).

  Let \( \vp=(p_1,\dots,p_n)\in X \) and \( \varphi(\vp)=T \). Since the west step \((a,b) \to
  (a-1,b)\) in \(p_i\) is the \( (\mu_i-i-a+1) \)st west step, it corresponds to
  the entry \( T(i,j)=b \), where \( j = \mu_i+1-(\mu_i-i-a+1)=i+a \), or
  equivalently, \(a=c(i,j)\). Therefore, under this correspondence, the weight
  of the west step \((a,b) \to (a-1,b)\) in \(p_i\) can be written as \(
  \alpha_{b}-\beta_{-a+b} = \alpha_{T(i,j)}-\beta_{T(i,j)-c(i,j)} \), and we
  obtain
\[
  \sum_{\vp\in X}\wt(\vp) = \sum_{T\in\IET_\ZZ(\ml)} \prod_{(i,j)\in\ml} (\alpha_{T(i,j)}-\beta_{T(i,j)-c(i,j)}).
\]
The theorem then follows from \eqref{eq:C=nonintersecting} and the above equation.
\end{proof}

\begin{figure}
  \centering
  
    \begin{tikzpicture}[scale=0.68, square/.style= {rectangle,draw=black!100,thin, minimum size= 6.8mm}]
      \draw[black!20] (-8,-5) grid (4,5);
      \draw[->, very thick, black!30] (-8,0) -- (4,0);
      \draw[->, very thick, black!30] (0,-5) -- (0,5);
      \draw[line width=1.5pt] 
      (3,1) \U \U \LStep{3}
      (2,1) \U\LStep{2}\LStep{2}
      (1,1) \LStep{1}\LStep{1} \U
      (-1,0) \LStep{0}\U\LStep{1} 
      (-2,-1) \LStep{-1}\U\LStep{0}\LStep{0} 
      (-3,-2) \LStep{-2}\LStep{-2}\U\LStep{-1}\U 
      (-5,-4) \U\LStep{-3}\U\LStep{-2}\U\U;
      \draw[dashed] (-6,-5) -- (0,1) -- (4,1);
      \draw[dotted, line width = 1pt] (3,1)--(3,-5.5);
      \node[below] at (3,-5.5) {\( \overline{u}_1 \)};
      \draw[dotted, line width = 1pt] (2,1)--(2,-5.5);
      \node[below] at (2,-5.5) {\( \overline{u}_2 \)};
      \draw[dotted, line width = 1pt] (1,1)--(1,-5.5);
      \node[below] at (1,-5.5) {\( \overline{u}_3 \)};
      \draw[dotted, line width = 1pt] (-1,0)--(-1,-5.5);
      \node[below] at (-1,-5.5) {\( \overline{u}_4 \)};
      \draw[dotted, line width = 1pt] (-2,-1)--(-2,-5.5);
      \node[below] at (-2,-5.5) {\( \overline{u}_5 \)};
      \draw[dotted, line width = 1pt] (-3,-2)--(-3,-5.5);
      \node[below] at (-3,-5.5) {\( \overline{u}_6 \)};
      \draw[dotted, line width = 1pt] (-5,-4)--(-5,-5.5);
      \node[below] at (-5,-5.5) {\( \overline{u}_7 \)};
      \draw[dotted, line width = 1pt] (2,3)--(2,5.5);
      \node[above] at (2,5.5) {\( \overline{v}_1 \)};
      \draw[dotted, line width = 1pt] (0,2)--(0,5.5);
      \node[above] at (0,5.5) {\( \overline{v}_2 \)};
      \draw[dotted, line width = 1pt] (-1,2)--(-1,5.5);
      \node[above] at (-1,5.5) {\( \overline{v}_3 \)};
      \draw[dotted, line width = 1pt] (-3,1)--(-3,5.5);
      \node[above] at (-3,5.5) {\( \overline{v}_4 \)};
      \draw[dotted, line width = 1pt] (-5,0)--(-5,5.5);
      \node[above] at (-5,5.5) {\( \overline{v}_5 \)};
      \draw[dotted, line width = 1pt] (-6,0)--(-6,5.5);
      \node[above] at (-6,5.5) {\( \overline{v}_6 \)};
      \draw[dotted, line width = 1pt] (-7,0)--(-7,5.5);
      \node[above] at (-7,5.5) {\( \overline{v}_7 \)};

      \filldraw[black] 
      (2,3) circle (3pt) node[above left] {\(v_1\)}
      (0,2) circle (3pt) node[above left] {\(v_2\)}
      (-1,2) circle (3pt) node[above left] {\(v_3\)}
      (-3,1) circle (3pt) node[above left] {\(v_4\)}
      (-5,0) circle (3pt) node[above left] {\(v_5\)}
      (-6,0) circle (3pt) node[above left] {\(v_6\)}
      (-7,0) circle (3pt) node[above left] {\(v_7\)}
      (3,1) circle (3pt) node[below right] {\(u_1\)}
      (2,1) circle (3pt) node[below right] {\(u_2\)}
      (1,1) circle (3pt) node[below right] {\(u_3\)}
      (-1,0) circle (3pt) node[below right] {\(u_4\)}
      (-2,-1) circle (3pt) node[below right] {\(u_5\)}
      (-3,-2) circle (3pt) node[below right] {\(u_6\)}
      (-5,-4) circle (3pt) node[below right] {\(u_7\)};
     
     \draw[<->, very thick] (4.5,-0.5)--(5.6,-0.5);

     \node[square] at (10,3) {\(3\)};
     \node[square] at (10,2) {\(2\)};
     \node[square] at (9,2) {\(2\)};
     \node[square] at (10,1) {\(1\)};
     \node[square] at (9,1) {\(1\)};
     \node[square] at (9,0) {\(0\)};
     \node[square] at (8,0) {\(1\)};
     \node[square] at (9,-1) {\(-1\)};
     \node[square] at (8,-1) {\(0\)};
     \node[square] at (7,-1) {\(0\)};
     \node[square] at (9,-2) {\(-2\)};
     \node[square] at (8,-2) {\(-2\)};
     \node[square] at (7,-2) {\(-1\)};
     \node[square] at (8,-3) {\(-3\)};
     \node[square] at (7,-3) {\(-2\)};

    \end{tikzpicture}
    \caption{A nonintersecting \( n \)-path \(
      \overline{\vp}=(\overline{p}_1,\dots,\overline{p}_n) \) with \(
      \overline{p}_i\in P_{\Gamma}(\overline{u}_i,\overline{v}_i) \), where
      \(u_i=(\mu_i-i,\min(\mu_i-i+1,1))\), \(v_i=(\lambda_i-i,\lambda_i)\),
      \(\mu=(4,4,4,3,3,3,2)\), \(\lambda=(3,2,2,1)\) and \( n=7 \), and the
      corresponding tableau \( T \). Since each \( \overline{p}_i \) passes
      through \(u_i=(\mu_i-i,\min(\mu_i-i+1,1))\) and
      \(v_i=(\lambda_i-i,\lambda_i)\), we have \( T\in\IET_\ZZ(\ml) \).}
  \label{fig:paths=IET}
\end{figure}

\section{Modified Jacobi--Trudi-like formulas} \label{sec:var JT}
In Section~\ref{sec:schur_expansion}, we gave Jacobi--Trudi-like formulas
(Theorem~\ref{thm:JT_ab}) for \( G_\lambda(\vx_n;\va,\vb) \) and \(
g_\lambda(\vx_n;\va,\vb) \). In this section, we present modifications of these
Jacobi--Trudi-like formulas. In later sections the modified formulas will give a connection between \(
G_\lambda(\vx_n;\va,\vb) \) (resp.~\( g_\lambda(\vx_n;\va,\vb) \)) and their
skew version  \( G_\lm(\vx;\va,\vb) \) (resp.~\( g_\lm(\vx;\va,\vb) \)).

We first prove a modification of the Jacobi--Trudi-like formula for \(
g_\lambda(\vx_n;\va,\vb) \) in Theorem~\ref{thm:JT_ab} because its proof
is simpler than that for \( G_\lambda(\vx_n;\va,\vb) \).

\begin{thm} \label{thm:det=det_g}
  For \( \lambda\in\Par_n \), we have
  \begin{equation}  \label{eq:det=det_g}
    g_\lambda(\vx_n;\va,\vb)
    = \det \left( h_{\lambda_i-i+j}[X_n-A_{\lambda_i-1}+B_{i-1}-B_{j-1}] \right)_{i,j=1}^n.
  \end{equation}
\end{thm}
\begin{proof}
  Let \( H_{i,j}=h_{\lambda_i-i+j}[X_n-A_{\lambda_i-1}+B_{i-1}] \).
  By Theorem~\ref{thm:JT_ab}, it suffices to show that 
  \begin{equation}
    \label{eq:hh}
    \det\left(H_{i,j}\right)_{i,j=1}^n
    =\det \left( h_{\lambda_i-i+j}[X_n-A_{\lambda_i-1}+B_{i-1}-B_{j-1}] \right)_{i,j=1}^n.
  \end{equation}
The \( (i,j) \)-entry of the matrix on the right hand side of \eqref{eq:hh} can be rewritten as follows:
\[
  \sum_{k\ge 0} h_{\lambda_i-i+j-k}[X_n-A_{\lambda_i-1}+B_{i-1}] h_k[-B_{j-1}] 
  = H_{i,j} + \sum_{k=1}^{j-1} H_{i,j-k} h_k[-B_{j-1}],
\]
  where \( k \) is at most \( j-1 \) in the sum on the right because
  if \( k\ge j \), then \( h_k[-B_{j-1}] = (-1)^k e_k[B_{j-1}]=0 \).
  Therefore the matrix on the right hand side of \eqref{eq:hh} can be obtained
  from the matrix on the left hand side by elementary column operations, which shows
  \eqref{eq:hh}.
\end{proof}
 
Now we state a modification of the Jacobi--Trudi-like formula for \(
G_\lambda(\vx_n;\va,\vb) \) in Theorem~\ref{thm:JT_ab}.

\begin{thm} \label{thm:det=det_G}
  For \( \lambda\in\Par_n \), we have
  \begin{equation} \label{eq:det=det_G}
    G_\lambda(\vx_n;\va,\vb) = C_n \det \left( h_{\lambda_i-i+j}[X_n\ominus (A_{\lambda_i}-B_{i-1}+B_j)] \right)_{i,j=1}^n,
  \end{equation}
  where
  \[
    C_n = \prod_{i,j=1}^n (1-\beta_i x_j).
  \]
\end{thm}

To prove Theorem~\ref{thm:det=det_G}, we first show some auxiliary lemmas.

\begin{lem}\label{lem:jtb}
  For \( \lambda\in\Par_n \), we have
  \[
    s_\lambda(\vx_n) = \det(h_{\lambda_i-i+j}[X_{n}])_{i,j=1}^n
    = \det(h_{\lambda_i-i+j}[X_{n-j+1}])_{i,j=1}^n.
  \]
\end{lem}
\begin{proof}
  The first equality is just the Jacobi--Trudi formula for Schur polynomials.
  To show the second equality, we use  a similar argument in the proof of
  Theorem~\ref{thm:det=det_g}.
  Let \( H_{i,j} =  h_{\lambda_i-i+j}[X_{n}] \).
  Then the \( (i,j) \)-entry of the matrix on the right hand side is
  \begin{align*}
    h_{\lambda_i-i+j}[X_{n-j+1}]
    &= h_{\lambda_i-i+j}[X_n+(X_{n-j+1}-X_n)] \\
    &= h_{\lambda_i-i+j}[X_n]
      + \sum_{k=1}^{j-1} h_{\lambda_i-i+j-k}[X_n] h_k[X_{n-j+1}-X_n] \\
    &= H_{i,j}
      + \sum_{k=1}^{j-1} H_{i,j-k} h_k[X_{n-j+1}-X_n],
  \end{align*}
  by a similar reason in the proof of Theorem~\ref{thm:det=det_g}.
  Hence, using elementary column operations we have the desired result.
\end{proof}

A \emph{reverse semistandard Young tableau} \( T \) of shape \( \lambda \) is
a filling of \( \lambda \) with positive integers such that
each row of \( T \) is weakly decreasing from left to right,
and each column of \( T \) is strictly decreasing from top to bottom.
Let \( \SSYT_n(\lambda) \) (resp.~\( \RSSYT_n(\lambda) \))
be the set of semistandard Young tableaux
(resp.~reverse semistandard Young tableaux) of
shape \( \lambda \) whose entries are at most \( n \). It is well known that 
\[
  s_\lambda(\vx_n) = \sum_{T\in\SSYT_n(\lambda)} \vx^T = \sum_{T\in\RSSYT_n(\lambda)} \vx^T.
\]

For \( \lambda,\mu\in \Par_n \), we define
\[
  \lambda\circ\mu := (\lambda_1+\mu_1,\lambda_1+\mu_2,\dots,\lambda_1+\mu_n)
  / (\lambda_1-\lambda_n,\lambda_1-\lambda_{n-1},\dots, \lambda_1-\lambda_1).
\]
In other words, \( \lambda\circ\mu \) is the skew shape obtained by 
rotating \( \lambda \) by 180 degrees and then attaching \( \mu \) to its right.
For example, if \( \lambda = (3,1,0), \mu=(4,2,2)\in\Par_3 \), then
\ytableausetup{smalltableaux}
\[
  \lambda\circ\mu = (7,5,5)/(3,2,0) = \vcenter{\hbox{\ydiagram{3+4,2+3,5}}}.
\]

\begin{lem}[{\cite[Exercise~7.34]{EC2}}] \label{lem:lambda_mu=lambda_circ_mu}
  For \( \lambda,\mu\in\Par_n \), we have
  \[
    s_\lambda(\vx_n) s_\mu(\vx_n) = s_{\lambda\circ\mu}(\vx_n).
  \]
\end{lem}
\begin{proof}
  We prove the lemma by giving a weight-preserving bijection \( \Phi \) between
  \( \RSSYT_n(\lambda)\times\SSYT_n(\mu) \) and \( \SSYT_n(\lambda\circ\mu) \).
  For \( (T_1, T_2)\in\RSSYT_n(\lambda)\times\SSYT_n(\mu) \),
  define \( \Phi(T_1,T_2) \) to be the tableau of shape \( \lambda\circ\mu \)
  obtained from \( T_2 \) by attaching the 180 degrees rotation of \( T_1 \).
  For example, let  \( \lambda = (3,1,0), \mu=(4,2,2)\in\Par_3 \),
  then \( \Phi \) maps
  \[
    \left(
      \vcenter{\hbox{
      \begin{ytableau}
      3 & 3 & 2 \\
      1
      \end{ytableau}}},
      \vcenter{\hbox{
      \begin{ytableau}
        1 & 1 & 2 & 3 \\
        2 & 2 \\
        3 & 3
      \end{ytableau}}}
    ~\right) \mapsto
    \vcenter{\hbox{
      \begin{ytableau}
        \none & \none & \none & 1 & 1 & 2 & 3 \\
        \none & \none & 1 & 2 & 2 \\
        2 & 3 & 3 & 3 & 3
      \end{ytableau}}}.
  \]
  By definition, for \( T\in\SSYT_n(\mu) \), the entries in the \( i \)th row are at least \( i \).
  Similarly, for \( T\in\RSSYT_n(\lambda) \), the entries in the \( i \)th row are at most \( n-i+1 \).
  From these facts, \( \Phi(T_1,T_2) \) belongs to \( \SSYT_n(\lambda\circ\mu) \).
  It is straightforward to check that \( \Phi \) is bijective and weight-preserving,
  i.e., for \( (T_1,T_2)\in \RSSYT_n(\lambda)\times\SSYT_n(\mu) \), \( \vx^{T_1} \vx^{T_2} = \vx^{\Phi(T_1,T_2)} \).
\end{proof}

\begin{lem} \label{lem:hXB=nu_circ_mu}
  For \( \mu\in\Par_n \),
  \[
    \det (h_{\mu_i-i+j}[X_n\ominus B_j])_{i,j=1}^n =
    s_\mu(\vx_n) \prod_{i,j=1}^n \frac{1}{1-\beta_i x_j}.
  \]
\end{lem}
\begin{proof}
  By definition, for \( 1\le i,j\le n \), we have
  \[
    h_{\mu_i-i+j}[X_n\ominus B_j] =
    \sum_{k\in \ZZ} h_{\mu_i-i+k+1}[X_n] h_{k+1-j}[B_j]
    =\sum_{k\ge0} h_{\mu_i-i+k+1}[X_n] h_{k+1-j}[B_j],
  \]
  which is the \( (i,j) \)-entry of the matrix \( PQ \), where
  \[
    P=(h_{\mu_i-i+j+1}[X_n])_{i\in[n],j\in \NN},
    \qquad Q=(h_{i+1-j}[B_j])_{i\in\NN,j\in [n]}.
  \]
  Thus, by Lemma~\ref{lem:C-B},
  \[
    \det (h_{\mu_i-i+j}[X_n\ominus B_j])_{i,j=1}^n = \det (PQ)
    =\sum_{\nu\in\Par_n} \det P^\nu \det Q_\nu.
  \]
  Observe that
  \begin{align*}
    \det P^\nu
    &= \det \left( h_{\mu_i-i+\nu_{j}+n-j+1}[X_n] \right)_{i,j=1}^n
    =(-1)^{n(n-1)/2} \det \left( h_{\mu_i-i+\nu_{n-j+1}+j}[X_n] \right)_{i,j=1}^n\\
    &= (-1)^{n(n-1)/2} s_{\nu\circ\mu}(\vx_n), \\
    \det Q_\nu
    &= \det(h_{\nu_i+n-i+1-j}[B_j])_{i,j=1}^n
    =(-1)^{n(n-1)/2} \det(h_{\nu_{i}-i+j}[B_{n-j+1}])_{i,j=1}^n\\
    &=(-1)^{n(n-1)/2} s_{\nu}(\vb_n),
  \end{align*}
  where \( \vb_n=(\beta_1,\dots,\beta_n) \), and the last equality follows from
  Lemma~\ref{lem:jtb}.
  Combining the above equations with Lemma~\ref{lem:lambda_mu=lambda_circ_mu},
  we obtain
  \[
    \det (h_{\mu_i-i+j}[X_n\ominus B_j])_{i,j=1}^n
      = \sum_{\nu\in\Par_n} s_{\nu\circ\mu}(\vx_n) s_\nu(\vb_n)
      = s_\mu(\vx_n) \sum_{\nu\in\Par_n} s_\nu(\vx_n) s_\nu(\vb_n).
  \]
  Therefore, the Cauchy identity
  \[
    \sum_{\nu\in\Par_n} s_\nu(\vx_n) s_\nu(\vb_n)
      = \prod_{i,j=1}^n \frac{1}{1-\beta_i x_j}
  \]
  completes the proof.
\end{proof}

Finally, we are ready to prove Theorem~\ref{thm:det=det_G}.

\begin{proof}[Proof of Theorem~\ref{thm:det=det_G}]
  Observe that for any \( m\in\ZZ \) and formal power series \( X,Y,Z \) without
  constant terms,
\[
  h_{m}[X\ominus(Y+Z)]
  =\sum_{a,b\ge 0} h_{m+a+b}[X] h_{a}[Y] h_{b}[Z]= h_{m}[(X\ominus Y)\ominus Z] .
\]
Thus the theorem can be restated as
\begin{equation}\label{eq:C_ndet}
  G_\lambda(\vx_n;\va,\vb) = C_n \det \left( h_{\lambda_i-i+j}[(X_n\ominus B_{j})\ominus (A_{\lambda_{i}}-B_{i-1})] \right)_{i,j=1}^n.
\end{equation} 

The transpose of the matrix in \eqref{eq:C_ndet} has \( (i,j) \)-entry
\begin{align}
  h_{\lambda_j-j+i}[(X_n\ominus B_{i})\ominus (A_{\lambda_{j}}-B_{j-1})] 
  &= \sum_{k\in\ZZ} h_{k}[X_n\ominus B_{i}] h_{k-(\lambda_j-j+i)}[A_{\lambda_{j}}-B_{j-1}] \nonumber\\
  &= \sum_{k\in\ZZ} h_{k-n+i}[X_n\ominus B_{i}] h_{k-(\lambda_j+n-j)}[A_{\lambda_{j}}-B_{j-1}]. \label{eq:shifted_CB}
\end{align}
Note that if \( k<0 \), then \( h_{k-(\lambda_j+n-j)}[A_{\lambda_{j}}-B_{j-1}]=0
\) because \( k-(\lambda_j+n-j)<0 \). Thus, \eqref{eq:shifted_CB} is equal to
the \( (i,j) \)-entry of the matrix \( PQ \), where
\[
  P=(h_{j-n+i}[X_n\ominus B_{i}])_{i\in [n],j\in\NN}, \quad
Q=(h_{i-(\lambda_j+n-j)}[A_{\lambda_j}-B_{j-1}])_{i\in \NN, j\in[n]}.
\]
Then, by Lemma~\ref{lem:C-B}, \eqref{eq:C_lm}, and
Lemma~\ref{lem:hXB=nu_circ_mu}, the right hand side of \eqref{eq:C_ndet} is
equal to
\begin{align*}
  & C_n\det \left( PQ \right) = C_n \sum_{\mu \in \Par_n} (\det P^{\mu}) (\det Q_{\mu})\\
  &= C_n\sum_{\mu\in\Par_n}
    \det (h_{\mu_j+i-j}[X_n\ominus B_i])_{i,j=1}^n
    \det (h_{\mu_i-\lambda_j-i+j}[A_{\lambda_j}-B_{j-1}])_{i,j=1}^n\\
  &= C_n \sum_{\mu\in\Par_n} C_{\lambda,\mu}(\va,\vb)
    \det (h_{\mu_j+i-j}[X_n\ominus B_i])_{i,j=1}^n\\
  &=  \sum_{\mu\in\Par_n} C_{\lambda,\mu}(\va,\vb) s_\mu(\vx_n),
\end{align*}
which is equal to the left hand side of \eqref{eq:C_ndet} by
\eqref{eq:G_n=s*det}. This completes the proof.
\end{proof}

\section{Schur expansions and the omega involution for skew shapes}
\label{sec:omega_involution}

In this section, similar to Section \ref{sec:schur_expansion}, we expand
\(G_\lm(\vx;\va,\vb)\) and \(g_\lm(\vx;\va,\vb)\) in terms of skew Schur
functions using Jacobi--Trudi-like formulas and the Cauchy--Binet theorem. As a
consequence we show that \( G_{\lm}(\vx;\va,\vb) \) and \(
g_\lm(\vx;\va,\vb) \) behave nicely under the involution \( \omega \). One
of our results (Theorem~\ref{thm:G=DSD}) in this section generalizes the
result of Chan and Pflueger \cite{CP21:grothendieck} on skew Schur expansions
of \( RG_\lm(\vx;\vb)=G_\lm(\vx;\bm0,\vb) \).

As mentioned in the introduction, one can define
\( G_\lm(\vx;\va,\vb) \) and \( g_\lm(\vx;\va,\vb) \) using
Theorem~\ref{thm:comb_intro}. In our companion paper
\cite{Hwang-preprint} we prove the following Jacobi--Trudi-like
formulas for \( G_\lm(\vx;\va,\vb) \) and \( g_\lm(\vx;\va,\vb) \).

\begin{thm}[{\cite{Hwang-preprint}}]
  \label{thm:skew_Gg_JT_formula}
  For \(\lambda,\mu\in \Par_n\), we have
  \begin{align}
    \label{eq:JT skew G} G_\lm(\vx;\va,\vb) 
      &= C \det\left(   h_{\lambda_i-\mu_j-i+j}
        [X\ominus(A_{\lambda_i}-A_{\mu_j}-B_{i-1} + B_{j})]\right)_{i,j=1}^n,\\
    \label{eq:dualJT skew G}  G_{\lambda'/\mu'}(\vx;\va,\vb) 
      &= D \det\left(   e_{\lambda_i-\mu_j-i+j}
      [X\ominus(A_{i-1}-A_{j}-B_{\lambda_i} + B_{\mu_j})]\right)_{i,j=1}^n,\\
    \label{eq:JT skew g}  g_\lm(\vx;\va,\vb)
      &= \det\left(   h_{\lambda_i-\mu_j-i+j}
        [X-A_{\lambda_i-1}+A_{\mu_j}+B_{i-1} - B_{j-1}]\right)_{i,j=1}^n,\\
    \label{eq:dualJT skew g}  g_{\lambda'/\mu'}(\vx;\va,\vb)
      &= \det\left(   e_{\lambda_i-\mu_j-i+j}
        [X-A_{i-1}+A_{j-1}+B_{\lambda_i-1} - B_{\mu_j}]\right)_{i,j=1}^n,
  \end{align}
  where
    \[
      C = \prod_{i=1}^{n} \prod_{l=1}^{\infty} (1-\beta_i x_l) \mbox{ and } 
      D = \prod_{i=1}^{n} \prod_{l=1}^{\infty} (1-\alpha_i x_l)^{-1}.
    \]
\end{thm}

In this paper, instead of introducing the combinatorial models in
Theorem~\ref{thm:comb_intro}, we will consider the formulas in
Theorem~\ref{thm:skew_Gg_JT_formula} as the definitions of
\( G_\lm(\vx;\va,\vb) \), \( G_{\lambda'/\mu'}(\vx;\va,\vb) \),
\( g_\lm(\vx;\va,\vb) \), and \( g_{\lambda'/\mu'}(\vx;\va,\vb) \).
These definitions are compatible with each other in the sense that
\( G_\lm(\vx;\va,\vb) = G_{(\lambda')'/(\mu')'}(\vx;\va,\vb) \) and
\( g_\lm(\vx;\va,\vb) = g_{(\lambda')'/(\mu')'}(\vx;\va,\vb) \), which
follows from the results of \cite{Hwang-preprint}. However, in this
paper, we will simply regard \( \lambda'/\mu' \) as a symbol so that
we do not have to check the well-definedness of these definitions. In
fact, the arguments in this section only need the right-hand sides of
\eqref{eq:JT skew G}, \eqref{eq:dualJT skew G}, \eqref{eq:JT skew g},
and \eqref{eq:dualJT skew g}.

  We define a \emph{generalized partition} of length \(n\) to be a sequence
  \((\lambda_1,\dots,\lambda_n)\in\ZZ^n\) satisfying \(\lambda_1\geq \dots\geq
  \lambda_n\). Denote by \( \GPar_n \) the set of generalized partitions of
  length \( n \). For \( \lambda,\mu\in\GPar_n \), we write \(
  \mu\subseteq\lambda \) if \( \mu_i\le \lambda_i \) for \( i=1,\dots,n \).
  Observe that this notation is consistent with the Young diagram inclusion for
  partitions. For \( \lambda,\mu\in\GPar_n \) with \( \mu\subseteq\lambda \),
  the \emph{generalized skew shape} \( \lambda/\mu \) is defined to be the set
  \( \{(i,j)\in\ZZ^2: 1\le i\le n~\mbox{and}~\mu_i<j\le \lambda_i\} \).
  Each element \( (i,j)\in\lm \) is called a \emph{cell} of \( \lm \) and
  \( |\lm| \) denotes the number of cells in \( \lm \). Similar to Young
  diagrams, we visualize \( \lm \) as an array of squares by placing a square in
  row \( i \) and column \( j \) for each cell \( (i,j)\in\lm \), where we may
  have columns with negative indices. Thus we can define a tableau of shape \(
  \lm \) as a filling of the cells in \( \lm \) as usual. For each cell \( (i,j)\in\lm
  \), we define \( c(i,j):=j-i \).
  We consider a partition \( \lambda\in \Par_n \) as an element of \( \GPar_n \).
  Then the above definitions on generalized partitions are consistent with those
  on partitions.
  
  We extend the definition of skew Schur functions \(
  s_{\lambda/\mu}(\vx) \) and \( 
  s_{\lambda'/\mu'}(\vx) \) for \( \lambda,\mu\in\GPar_n \) by
  the Jacobi--Trudi formulas:
  \[
    s_{\lambda/\mu}(\vx)
    = \det \left( h_{\lambda_i-\mu_j-i+j}(\vx) \right)_{i,j=1}^n, \qquad
    s_{\lambda'/\mu'}(\vx)
    = \det \left( e_{\lambda_i-\mu_j-i+j}(\vx) \right)_{i,j=1}^n.
  \]
  (As before, we regard \( \lambda'/\mu' \) as a symbol, and
  \( \lambda' \) and \( \mu' \) will not be considered individually.
  Hence, again, it is not necessary to check the well-definedness of
  the above definitions of \( s_{\lambda/\mu}(\vx) \) and
  \( s_{\lambda'/\mu'}(\vx) \) for \( \lambda,\mu\in\GPar_n \).)
  Clearly, this definition gives the usual skew Schur functions if
  \( \lambda,\mu\in\Par_n \). By Lemma~\ref{lem:det(h)=0} (with
  \( \GPar_n \) in place of \( \Par_n \)), we have
  \( s_{\lambda/\mu}(\vx)=s_{\lambda'/\mu'}(\vx)=0 \) unless
  \( \mu\subseteq\lambda \). Moreover, if \( \mu\subseteq\lambda \),
  then \( \lambda-\mu_n,\mu-\mu_n\in\Par_n \), where
  \( \lambda-k=(\lambda_1-k,\lambda_2-k,\dots,\lambda_n-k) \), and
  \[
    s_{\lambda/\mu}(\vx) = s_{(\lambda-\mu_n)/(\mu-\mu_n)}(\vx), \qquad
    s_{\lambda'/\mu'}(\vx) = s_{(\lambda-\mu_n)'/(\mu-\mu_n)'}(\vx).
  \]
  Therefore, \( s_\lm(\vx) \) is equal to a usual skew Schur function for any \( \lambda,\mu\in\GPar_n \).

  We extend the notation in Section~\ref{sec:schur_expansion} as follows.
  For \(\lambda,\mu\in\GPar_n\), \(P=(p_{i,j})_{i\in [n], j\in\ZZ}\),
  \(Q=(q_{i,j})_{i\in\ZZ,j\in [n]}\), and \(R=(r_{i,j})_{i,j\in\ZZ}\), we define
  \begin{align*}
    P^{\lambda}&:=(p_{i,\lambda_j+n-j})_{1\le i,j \le n}, \\
    Q_{\lambda}&:=(q_{\lambda_i+n-i,j})_{1\le i,j \le n}, \\
    R_\lambda^\mu&:=(r_{\lambda_i+n-i,\mu_j+n-j})_{1\le i,j \le n}.
  \end{align*}

  The following lemma can be proved similarly as Lemma~\ref{lem:C-B}.
  
  \begin{lem}[Cauchy--Binet theorem for generalized partitions]
    \label{lem: C-B for GPar}
    Suppose that
    \(P=(p_{i,j})_{i\in [n],j\in \ZZ}\) and \(Q=(q_{i,j})_{i\in \ZZ,j\in [n]}\)
    are matrices such that each entry in \( PQ \) is well defined as a series.
    Then we have
    \[
      \det \left( P Q \right) = \sum_{\nu\in\GPar_n} (\det P^{\nu}) (\det Q_{\nu}).
    \]
  \end{lem}

  Applying Lemma~\ref{lem: C-B for GPar} twice and using the fact
  \( (QR)_\nu = Q_\nu R \), we obtain the following corollary.

  \begin{cor}
    \label{cor:cor of CB}
    Suppose that
    \(P=(p_{i,j})_{i\in [n],j\in \ZZ}\), \(Q=(q_{i,j})_{i,j\in \ZZ}\)
    and \( R=(r_{i,j})_{i\in\ZZ, j\in [n]} \)
    are matrices such that each entry in \( PQ, QR \) and \( PQR \) is well defined
    as a series. Then we have
    \begin{align*}
      \det \left( P Q R \right) = \sum_{\nu,\rho\in\GPar_n} (\det P^{\nu}) (\det Q_{\nu}^{\rho}) (\det R_{\rho}).
    \end{align*}
  \end{cor}
  
  We are now ready to expand \( G_\lm(\vx;\va,\vb) \) and \( G_{\lambda'/\mu'}(\vx;\va,\vb)
  \) in terms of skew Schur functions up to some factor.

  \begin{thm} \label{thm:G=CsC}
    For \(\lambda,\mu\in \Par_n\) with \( \mu\subseteq\lambda \), we have
    \begin{align} \label{eq:G=CsC}
      G_\lm(\vx;\va,\vb)= C
      \sum_{\substack{\rho,\nu\in \GPar_n \\ \rho\subseteq\mu\subseteq\lambda\subseteq\nu}}
      C_{\lambda,\nu}(\va,\vb) s_{\nu/\rho}(\vx) C'_{\rho,\mu}(\va,\vb),
    \end{align}
    where 
    \begin{align*}
      C &= \prod_{i=1}^{n} \prod_{l=1}^{\infty} (1-\beta_i x_l), \\
      C_{\lambda,\nu}(\va,\vb) &= \det \left( h_{\nu_i-\lambda_j-i+j}[A_{\lambda_j}-B_{j-1}] \right)_{i,j=1}^n, \\
      C'_{\rho,\mu}(\va,\vb) &= \det \left( h_{\mu_i-\rho_j-i+j}[-A_{\mu_i}+B_i] \right)_{i,j=1}^n.
    \end{align*}
  \end{thm}
  \begin{proof}
    We rewrite the \((i,j)\)-entry of the matrix in \eqref{eq:JT skew G} as
    \begin{align*}
     & h_{\lambda_i-\mu_j-i+j}[X\ominus(A_{\lambda_i}-A_{\mu_j}-B_{i-1} + B_{j})] \\
      &= \sum_{a-b-c=\lambda_i-\mu_j-i+j}  h_{a}[X] h_{b}[A_{\lambda_i}-B_{i-1}] h_{c}[-A_{\mu_j}+B_j]\\
      &= \sum_{k,l \in \ZZ} h_{k-(\lambda_i+n-i)}[A_{\lambda_i}-B_{i-1}] h_{k-l}[X] h_{(\mu_j+n-j)-l}[-A_{\mu_j}+B_j].
    \end{align*}
    Let \(P=(h_{j-(\lambda_i+n-i)}[A_{\lambda_i}-B_{i-1}])_{i\in [n],j\in \ZZ },
    Q=(h_{i-j}[X])_{i,j\in\ZZ}\) and
    \(R=(h_{(\mu_j+n-j)-i}[-A_{\mu_j}+B_j])_{i\in \ZZ,j\in [n]} \). By the above
    equation and \eqref{eq:JT skew G}, we have \( G_\lm(\vx;\va,\vb)= C \det(PQR) \). On the other hand, by Corollary~\ref{cor:cor of CB} we have
    \begin{align*}
      \det \left( P Q R \right) = \sum_{\nu,\rho\in\GPar_n} (\det P^{\nu}) (\det Q_{\nu}^{\rho}) (\det R_{\rho})
      = \sum_{\nu,\rho\in\GPar_n} C_{\lambda,\nu}(\va,\vb) s_{\nu/\rho}(\vx) C'_{\rho,\mu}(\va,\vb).
    \end{align*}
    Then the proof follows since \( C_{\lambda,\nu}(\va,\vb) s_{\nu/\rho}(\vx) C'_{\rho,\mu}(\va,\vb) \)
     is zero unless \(\rho\subseteq\mu\subseteq\lambda\subseteq\nu\)
    by Lemma~\ref{lem:det(h)=0} (with \( \GPar_n \) in place of \( \Par_n \)).
  \end{proof}
  
  \begin{thm} \label{thm:G=DSD}
    For \(\lambda,\mu\in \Par_n\) with \( \mu\subseteq\lambda \), we have
   \begin{align*}
     G_{\lambda'/\mu'}(\vx;\va,\vb)= D
     \sum_{\substack{\rho,\nu\in \GPar_n\\ \rho\subseteq\mu\subseteq\lambda\subseteq\nu}}
     D_{\lambda,\nu}(\va,\vb) s_{\nu'/\rho'}(\vx) D'_{\rho,\mu}(\va,\vb),
  \end{align*}
  where 
  \begin{align*}
    D &= \prod_{i=1}^{n} \prod_{l=1}^{\infty} (1-\alpha_i x_l)^{-1}, \\
    D_{\lambda,\nu}(\va,\vb) &= \det \left( e_{\nu_i-\lambda_j-i+j}[A_{j-1}-B_{\lambda_j}] \right)_{i,j=1}^n, \\
    D'_{\rho,\mu}(\va,\vb) &= \det \left( e_{\mu_i-\rho_j-i+j}[-A_i+B_{\mu_i}] \right)_{i,j=1}^n.
  \end{align*}
  \end{thm}
  \begin{proof}
    This is similar to the proof of Theorem~\ref{thm:G=CsC}.
  \end{proof}

  \begin{remark}
    The current formulas for \( G_\lm(\vx;\va,\vb) \) and \(
    G_{\lambda'/\mu'}(\vx;\va,\vb) \) in Theorems~\ref{thm:G=CsC} and
    \ref{thm:G=DSD} are not yet linear combinations of skew Schur functions,
    since there are variables \(x_1,x_2,\dots\) in \(C\) and \(D\). Since
    \[
      C=\prod_{i=1}^n \sum_{m\geq 0}(-1)^m \beta_i^m e_m(\vx) \qand D= \prod_{i=1}^n \sum_{m\geq 0}\alpha_i^m h_m(\vx), 
    \] 
    it is possible to rewrite the formulas as linear combinations of
    skew Schur functions using the Pieri rule for skew shapes \cite{assaf2011pieri}. 
  \end{remark}
  
  Recall that \( \IET_{\ZZ}(\nu/\lambda) \) is the set of \( \ZZ \)-inelegant
  tableaux of shape \( \nu/\lambda \) defined in Section~\ref{sec:lattice}. We
  also define \( \overline{\ET}_{\ZZ}(\mu/\rho) \) to be the set of fillings \(
  T \) of the cells in the generalized skew shape \( \mu/\rho \) with integers such that the rows are
  weakly increasing, the columns are strictly increasing, and \( \min(i-j,0) <
  T(i,j)\le i \) for all \( (i,j)\in \mu/\rho \).

  The determinants \(C_{\lambda,\nu}(\va,\vb),C'_{\rho,\mu}(\va,\vb),
  D_{\lambda,\nu}(\va,\vb)\), and \( D'_{\rho,\mu}(\va,\vb) \) in
  Theorems~\ref{thm:G=CsC} and \ref{thm:G=DSD} have the following combinatorial
  interpretations.

  \begin{prop}\label{prop:CD}
    For \( \lambda,\mu\in\Par_n \) and \( \rho,\nu\in\GPar_n \)
    with \( \rho\subseteq\mu\subseteq\lambda\subseteq\nu \),
    we have
  \begin{align}
    \label{eq:C}
    C_{\lambda,\nu}(\va,\vb)
    &=\sum_{T\in\IET_{\ZZ}(\nu/\lambda)}\prod_{(i,j)\in \nu/\lambda}(\alpha_{T(i,j)}-\beta_{T(i,j)-c(i,j)}), \\
    \label{eq:C'}
    C'_{\rho,\mu}(\va,\vb)
    &=\sum_{T\in\overline{\ET}_{\ZZ}(\mu/\rho)}\prod_{(i,j)\in \mu/\rho}(-\alpha_{T(i,j)+c(i,j)}+\beta_{T(i,j)}),\\
    \label{eq:D}
    D_{\lambda,\nu}(\va,\vb)
    &=\sum_{T\in\IET_{\ZZ}(\nu/\lambda)}\prod_{(i,j)\in \nu/\lambda}(-\beta_{T(i,j)}+\alpha_{T(i,j)-c(i,j)}), \\
    \label{eq:D'}
    D'_{\rho,\mu}(\va,\vb)
    &=\sum_{T\in\overline{\ET}_{\ZZ}(\mu/\rho)}\prod_{(i,j)\in \mu/\rho}(\beta_{T(i,j)+c(i,j)}-\alpha_{T(i,j)}).
  \end{align}
  Here we assume \( \alpha_k=\beta_k=0 \) if \( k\le 0 \).
  \end{prop}
  \begin{proof}
    Since \(h_m[Z]=(-1)^m e_m[-Z]\) for any formal power series \(Z\), we have
    \begin{align*}
      \det \left( h_{\nu_i-\lambda_j-i+j}[A_{\lambda_j}-B_{j-1}] \right)_{i,j=1}^n &= 
      (-1)^{|\nu/\lambda|}\det \left( e_{\nu_i-\lambda_j-i+j}[B_{j-1}-A_{\lambda_j}] \right)_{i,j=1}^n \mbox{ and}\\
      \det \left( h_{\mu_i-\rho_j-i+j}[-A_{\mu_i}+B_i] \right)_{i,j=1}^n &=
      (-1)^{|\mu/\rho|} \det \left( e_{\mu_i-\rho_j-i+j}[-B_i+A_{\mu_i}] \right)_{i,j=1}^n.
    \end{align*}
    In addition, substituting \(-\vb\) and \(-\va\) for \(\va\) and \(\vb\),
    respectively, turns \( e_m[B_r-A_s] \) into \( (-1)^m e_m[A_r-B_s] \).
    Therefore, we obtain
    \begin{equation}
      \label{eq:CDCD}
      C_{\lambda,\nu}(\va,\vb)=D_{\lambda,\nu}(-\vb,-\va) \qand
      C'_{\rho,\mu}(\va,\vb)=D'_{\rho,\mu}(-\vb,-\va).
    \end{equation}

    Using \eqref{eq:CDCD}, we immediately obtain \eqref{eq:D} and \eqref{eq:D'} from \eqref{eq:C} and \eqref{eq:C'}, respectively.
    Since \eqref{eq:C} is the same as Theorem~\ref{thm:C_lm}, it remains to show \eqref{eq:C'}.
    
    The proof is similar to that of Theorem~\ref{thm:C_lm}. Let \( \Gamma \) be the weighted lattice graph in Proposition~\ref{prop: hm[-A+B]}.
    For any \( i, j \in \{1,\dots,n\} \), Proposition~\ref{prop: hm[-A+B]} with \( r=\mu_i-i, s = i \), and \( t = \rho_j-j \) gives 
    \begin{align*}
      h_{\mu_i-\rho_j-i+j}[-A_{\mu_i}+B_i] = \wt(P_{\Gamma} (( \rho_j-j,\min(-\rho_j+j-1,0) ),( \mu_i-i,i ))).
    \end{align*}
    If the first step of a path \( p\in P_{\Gamma} (( \rho_j-j,\min(-\rho_j+j-1,0) ),( \mu_i-i,i )) \) is an east step, then the weight of the step is zero since \( \rho_j-j+\min(-\rho_j+j-1,0)+1 \le 0 \) and \( \min(-\rho_j+j-1,0) \le 0 \).
    Hence, we can write 
    \begin{align*}
      h_{\mu_i-\rho_j-i+j}[-A_{\mu_i}+B_i] = \wt(P_{\Gamma} (( \rho_j-j,\min(-\rho_j+j,1) ),( \mu_i-i,i ))).
    \end{align*}

    By the Lindstr\"om--Gessel--Viennot lemma, \( \det\left( h_{\mu_i-\rho_j-i+j}[-A_{\mu_i}+B_i] \right)_{i,j=1}^n \) can be expressed as the weight sum of the nonintersecting \( n \)-paths \(\vp=(p_1,\dots,p_n)\) where \(p_i\in P_\Gamma((\rho_i-i,\min(-\rho_i+i,1)),(\mu_i-i,i))\), and the set of such nonintersecting \( n \)-paths \( \vp \) corresponds to \( \overline{\ET}_\ZZ(\mu/\rho) \) bijectively.
    Hence, we obtain \eqref{eq:C'}, which completes the proof.
  \end{proof}

  We show that the special case \( \va=(0,0,\dots) \) of Theorem~\ref{thm:G=DSD}
  is equivalent to (the transposed version of) the result of Chan and Pflueger
  \cite{CP21:grothendieck} on a skew Schur expansion of \(
  RG_\lm(\vx;\vb)=G_\lm(\vx;\bm0,\vb) \). We need some definitions to state
  their result.
  
  For \( \lambda,\mu\in\Par \), let \( X(\lm) \) denote the set of tableaux \( T
  \) of shape \( \lambda/\mu \) with integer entries such that the rows are
  weakly decreasing, the columns are strictly decreasing, and \( 0<T(i,j)<j \)
  for all \( (i,j)\in \lambda/\mu \). Let \( Y(\lm) \) denote the set of
  tableaux \( T \) of shape \( \lambda/\mu \) with integer entries  such that the
  rows are strictly increasing, the columns are weakly increasing, and \(
  0<T(i,j)\le j \) for all \( (i,j)\in \lambda/\mu \). For a tableau \( T \) of
  shape \( \lm \) with integer entries, let \(
  \vb^T=\prod_{(i,j)\in\lm}\beta_{T(i,j)} \) and \(
  (-\vb)^T=\prod_{(i,j)\in\lm}(-\beta_{T(i,j)}) \).

  Let \( \lambda,\mu\in\Par \) and fix an integer \( n \) satisfying \(
  \lambda,\mu\in\Par_n \). After taking transposes, the result of Chan and
  Pflueger \cite[Theorem~3.4]{CP21:grothendieck} can be stated as
  \begin{align}\label{eq:ChanPflueger}
    G_{\lambda'/\mu'}(\vx;\bm0,\vb)=\sum_{\substack{\rho,\nu\in\Par\\ \rho\subseteq\mu\subseteq\lambda\subseteq\nu}}
    \sum_{\substack{T_1\in X(\nu/\lambda)\\ T_2\in Y(\mu/\rho)}} (-\vb)^{T_1} s_{\nu'/\rho'}(\vx) \vb^{T_2}.
  \end{align}
  Combining Theorem~\ref{thm:G=DSD} for \( \va=(0,0,\dots) \) with \eqref{eq:D} and \eqref{eq:D'}, we have
  \begin{align}\label{eq:ChanPflueger3}
    G_\lmc(\vx;\bm0,\vb)=\sum_{\substack{\rho,\nu\in\GPar_n\\ \rho\subseteq\mu\subseteq\lambda\subseteq\nu}}D_{\lambda,\nu}(\bm0,\vb)s_{\nu'/\rho'}(\vx) D'_{\rho,\mu}(\bm0,\vb),
  \end{align}
  where 
  \begin{align*}
    D_{\lambda,\nu}(\bm0,\vb)&=\sum_{T\in\IET_{\ZZ}(\nu/\lambda)}\prod_{(i,j)\in \nu/\lambda}(-\beta_{T(i,j)}),\\
    D'_{\rho,\mu}(\bm0,\vb)&=\sum_{T\in\overline{\ET}_{\ZZ}(\mu/\rho)}\prod_{(i,j)\in \mu/\rho}\beta_{T(i,j)+c(i,j)}.
  \end{align*}

  Let \( S\in \overline{\ET}_{\ZZ}(\mu/\rho)\) for some \( \rho\in\GPar_n \)
  with \( \rho \subseteq \mu \). If there is a cell \( (i,j)\in\mu/\rho \)
  such that \( j\le 0 \), then \( S(i,j)+c(i,j) \le i +(j-i) \le 0 \). Thus, \(
  D'_{\rho,\mu}(\bm0,\vb) \) is zero unless \( \rho\in \Par_n \), and we can
  replace the condition \( \rho\in\GPar_n \) by \( \rho\in\Par_n \) in the sum
  in \eqref{eq:ChanPflueger3}. Note that \( \lambda\subseteq\nu \) implies \( \nu\in\Par_n \).
  On the other hand, if \( \ell(\nu)>\ell(\lambda)
  \), then by definition we have \( X(\nu/\lambda)=\emptyset \), which implies
  that we can replace the condition \( \rho,\nu\in\Par \) by \(
  \rho,\nu\in\Par_n \) in \eqref{eq:ChanPflueger}. Therefore, for the equivalence of
  \eqref{eq:ChanPflueger} and \eqref{eq:ChanPflueger3} it suffices to show that
  for \( \rho,\nu\in\Par_n \) with \(
  \rho\subseteq\mu\subseteq\lambda\subseteq\nu \) we have
  \begin{align}
    \label{eq:eqCP1}
    \sum_{T\in\IET_{\ZZ}(\nu/\lambda)}(-\vb)^T
    &= \sum_{T_1\in X(\nu/\lambda)} (-\vb)^{T_1},\\
    \label{eq:eqCP2}
    \sum_{T\in\overline{\ET}_{\ZZ}(\mu/\rho)}\vb^{T^*}
    &= \sum_{T_2\in Y(\mu/\rho)}  \vb^{T_2},
  \end{align}
  where for \( T\in\overline{\ET}_{\ZZ}(\mu/\rho) \)
  we define \( T^* \) to be the tableau of shape \( \mu/\rho \) 
  given by \( T^*(i,j)=T(i,j)+c(i,j) \) for all \( (i,j)\in\mu/\rho \).
  It is straightforward to check that
  \begin{align*}
    X(\nu/\lambda)
    &=\{T\in\IET_\ZZ(\nu/\lambda): T(i,j)>0 \mbox{ for all \( (i,j)\in \nu/\lambda\)}\},\\
    Y(\mu/\rho)
    &=\{T^*: T\in\overline{\ET}_{\ZZ}(\mu/\rho), T^*(i,j)>0 \mbox{ for all \( (i,j)\in \mu/\rho\)}\}.
  \end{align*}
  Since \( \beta_k=0 \) for \( k\le0 \), the above descriptions for \(
  X(\nu/\lambda) \) and \( Y(\mu/\rho) \) imply \eqref{eq:eqCP1} and
  \eqref{eq:eqCP2}, completing the proof of the equivalence of
  \eqref{eq:ChanPflueger} and \eqref{eq:ChanPflueger3}.

  We now show the effect of  the involution \(\omega\) on \(G_\lm(\vx;\va,\vb)\).

  \begin{thm}\label{thm:omega_G}
    For \(\lambda,\mu\in \Par\) with \( \mu\subseteq\lambda \), we have
    \begin{align*}
      \omega(G_\lm(\vx;\va,\vb)) &= G_{\lambda'/\mu'}(\vx;-\vb,-\va).
    \end{align*}
  \end{thm}
  \begin{proof}
    Fix an integer \( n \) such that \(\lambda,\mu\in \Par_n\).
    Let \(C(\vb)=\prod_{i=1}^{n} \prod_{l=1}^{\infty} (1-\beta_i x_l)\) and 
    \(D(\va)=\prod_{i=1}^{n} \prod_{l=1}^{\infty} (1-\alpha_i x_l)^{-1}\).
    By the well-known property of the involution \( \omega \) (see \cite[Lemma~7.14.4]{EC2}), we have
    \begin{align}\label{eq:omega_C=D}
      \omega(C(\vb))=\omega\left(\prod_{i=1}^{n}\prod_{l=1}^{\infty} (1-\beta_i x_l)\right)=\prod_{i=1}^{n} \prod_{l=1}^{\infty} (1+\beta_i x_l)^{-1}=D(-\vb).
    \end{align}
    Thus, using Theorem~\ref{thm:G=CsC}, \eqref{eq:CDCD}, and Theorem~\ref{thm:G=DSD}, we have
    \begin{align*}
      \omega(G_\lm(\vx;\va,\vb))
      &= \omega (C(\vb)) \sum_{\substack{\rho,\nu\in \GPar_n \\ \rho\subseteq\mu\subseteq\lambda\subseteq\nu}} C_{\lambda,\nu}(\va,\vb) \omega(s_{\nu/\rho}(\vx)) C'_{\rho,\mu}(\va,\vb) \\
      &= D(-\vb) \sum_{\substack{\rho,\nu\in \GPar_n \\ \rho\subseteq\mu\subseteq\lambda\subseteq\nu}} D_{\lambda,\nu}(-\vb,-\va) s_{\nu'/\rho'}(\vx) D'_{\rho,\mu}(-\vb, -\va) \\
      &= G_{\lambda'/\mu'}(\vx;-\vb, -\va),
    \end{align*}
    as desired.
  \end{proof}

  For the remaining part of this section, we expand \(g_\lm(\vx;\va,\vb)\) in
  terms of skew Schur functions and compute \(\omega(g_\lm(\vx;\va,\vb))\).

\begin{thm} \label{thm:g=csc}
  For \(\lambda,\mu\in \Par_n\), we have
  \begin{align*}
    g_\lm(\vx;\va,\vb) = \sum_{\substack{\rho,\nu\in \Par_n \\ \mu\subseteq\rho\subseteq\nu\subseteq\lambda}} c_{\lambda,\nu}(\va,\vb) s_{\nu/\rho}(\vx) c'_{\rho,\mu}(\va,\vb),
  \end{align*}
  where
  \begin{align*}
    c_{\lambda,\nu}(\va,\vb) &= \det\left( h_{\lambda_i-\nu_j-i+j}[-A_{\lambda_i-1}+B_{i-1}]  \right)_{i,j=1}^n ,  \\
    c'_{\rho,\mu}(\va,\vb) &= \det\left( h_{\rho_i-\mu_j-i+j}[A_{\mu_j}-B_{j-1}]  \right)_{i,j=1}^n.
  \end{align*}
\end{thm}
\begin{proof}
  The \((i,j)\)-entry of the matrix in \eqref{eq:JT
    skew g} can be rewritten as
  \begin{multline*}
    h_{\lambda_i-\mu_j-i+j}[X-A_{\lambda_i-1}+A_{\mu_j}+B_{i-1} - B_{j-1}] \\
    = \sum_{k,l\geq 0}h_{\lambda_i+n-i-k}[-A_{\lambda_i-1}+B_{i-1}] h_{k-l}[X] h_{l-(\mu_j+n-j)}[A_{\mu_j}-B_{j-1}].
  \end{multline*}
  Let \(P=(h_{\lambda_i+n-i-j}[-A_{\lambda_i-1}+B_{i-1}])_{i\in [n],j\in\NN}\), \(Q=(h_{i-j}[X])_{i,j\in\NN} \), and \(R=(h_{i-(\mu_j+n-j)}[A_{\mu_j}-B_{j-1}])_{i\in\NN,j\in [n]} \). 
  By the same argument used to obtain Corollary~\ref{cor:cor of CB} from
  Lemma~\ref{lem: C-B for GPar}, we can deduce from Lemma~\ref{lem:C-B} that
  \begin{align} \label{eq:C-B for PQR}
    \notag\det (PQR) =\sum_{\rho,\nu\in \Par_n} (\det P^{\nu}) (\det Q_{\nu}^{\rho}) (\det R_{\rho})
    =\sum_{\rho,\nu\in \Par_n} c_{\lambda,\nu}(\va,\vb) s_{\nu/\rho}(\vx) c'_{\rho,\mu}(\va,\vb).
  \end{align} 
  Then the proof follows from Lemma~\ref{lem:det(h)=0}.
\end{proof}

\begin{thm}\label{thm:g=dsd}
  For \(\lambda,\mu\in \Par_n\) with \( \mu\subseteq\lambda \), we have
  \begin{align*}
    g_{\lambda'/\mu'}(\vx;\va,\vb) = \sum_{\substack{\rho,\nu\in \Par_n \\ \mu\subseteq\rho\subseteq\nu\subseteq\lambda}} d_{\lambda,\nu}(\va,\vb) s_{\nu'/\rho'}(\vx) d'_{\rho,\mu}(\va,\vb),
  \end{align*}
  where
  \begin{align*}
    d_{\lambda,\nu}(\va,\vb) &= \det\left( e_{\lambda_i-\nu_j-i+j}[-A_{i-1}+B_{\lambda_i-1}]  \right)_{i,j=1}^n ,  \\
    d'_{\rho,\mu}(\va,\vb) &= \det\left( e_{\rho_i-\mu_j-i+j}[A_{j-1}-B_{\mu_j}]  \right)_{i,j=1}^n.
  \end{align*}
\end{thm}
\begin{proof}
  This is similar to the proof of Theorem~\ref{thm:g=csc}.
\end{proof}

The determinants \(c_{\lambda,\nu}(\va,\vb),c'_{\rho,\mu}(\va,\vb),
  d_{\lambda,\nu}(\va,\vb)\), and \( d'_{\rho,\mu}(\va,\vb) \) in
  Theorems~\ref{thm:g=csc} and \ref{thm:g=dsd} also have the following combinatorial
  interpretations.

\begin{prop}\label{prop:cd}
  For \( \lambda,\mu,\nu,\rho\in \Par_n \), we have
\begin{align*}
  c_{\lambda,\nu}(\va,\vb)
  &=\sum_{T\in\ET_{\ZZ}(\lambda/\nu)}\prod_{(i,j)\in \lambda/\nu}(-\alpha_{T(i,j)+c(i,j)}+\beta_{T(i,j)}), \\
  c'_{\rho,\mu}(\va,\vb)
  &=\sum_{T\in\IET_{\ZZ}(\rho/\mu)}\prod_{(i,j)\in \rho/\mu}(\alpha_{T(i,j)}-\beta_{T(i,j)-c(i,j)}),\\
  d_{\lambda,\nu}(\va,\vb)
  &=\sum_{T\in\ET_{\ZZ}(\lambda/\nu)}\prod_{(i,j)\in \lambda/\nu}(\beta_{T(i,j)+c(i,j)}-\alpha_{T(i,j)}), \\
  d'_{\rho,\mu}(\va,\vb)
  &=\sum_{T\in\IET_{\ZZ}(\rho/\mu)}\prod_{(i,j)\in \rho/\mu}(-\beta_{T(i,j)}+\alpha_{T(i,j)-c(i,j)}).
\end{align*}
\end{prop}
\begin{proof}
  By a similar argument in the proof of Proposition~\ref{prop:CD}, we have
  \begin{equation}
    \label{eq:cdcd}
    c_{\lambda,\nu}(\va,\vb) = d_{\lambda,\nu}(-\vb,-\va), \qquad
    c'_{\rho,\mu}(\va,\vb) = d'_{\rho,\mu}(-\vb,-\va).
  \end{equation}
  Since 
  \begin{align*}
    c'_{\rho,\mu}(\va,\vb)=\det\left( h_{\rho_i-\mu_j-i+j}[A_{\mu_j}-B_{j-1}]  \right)_{i,j=1}^n = C_{\mu,\rho}(\va,\vb),
  \end{align*}
  the proof follows from \eqref{eq:cdcd} and Theorems~\ref{thm:C_lm} and \ref{thm:c_lm}.
\end{proof}

\begin{thm}\label{thm:omega_g}
  For \(\lambda,\mu\in \Par\) with \( \mu\subseteq\lambda \), we have
  \begin{align*}
    \omega(g_\lm(\vx;\va,\vb)) &= g_{\lambda'/\mu'}(\vx;-\vb,-\va).
  \end{align*}
\end{thm}
\begin{proof}
  This can be proved similarly to the proof of Theorem~\ref{thm:omega_G} using
  \eqref{eq:cdcd} and Theorems~\ref{thm:g=csc} and \ref{thm:g=dsd}.
\end{proof}

\begin{remark}
  The involution \( \omega \) applied to \( G_{\lm}(\vx;\va,\vb) \) and \(
    g_\lm(\vx;\va,\vb) \) can also be computed directly using the Jacobi--Trudi-like formulas \eqref{eq:JT skew G}, \eqref{eq:dualJT skew G}, \eqref{eq:JT skew g}, and \eqref{eq:dualJT skew g},
    which gives simple proofs of Theorems~\ref{thm:omega_G} and \ref{thm:omega_g}.
  First we observe that for any \( n,a,b,c,d\in\ZZ \),
  \begin{align*}
    h_n[A_a-A_b-B_c+B_d] = (-1)^n e_n[-A_a+A_b+B_c-B_d] = e_n[-A_a+A_b+B_c-B_d]\Big\rvert_{\substack{\va = -\va \\ \vb=-\vb}}.
  \end{align*}
  Using this, for \( i,j\in \{1,\dots,n\} \) and \( m\in\ZZ \), we have
  \begin{align*}
     \omega( h_m[X\ominus(A_{\lambda_i}-A_{\mu_j}-B_{i-1} + B_{j})]) 
    &= \sum_{k\ge 0}\omega \left( h_{m+k}[X]\right) h_k[A_{\lambda_i}-A_{\mu_j}-B_{i-1} + B_{j}] \\
    &= \sum_{k\ge 0}e_{m+k}[X]e_k[-A_{\lambda_i}+A_{\mu_j}+B_{i-1} - B_{j}]\Big\rvert_{\substack{\va = -\va \\ \vb=-\vb}}\\
    &= e_m[X\ominus(B_{i-1} - B_{j}-A_{\lambda_i}+A_{\mu_j})]\Big\rvert_{\substack{\va = -\va \\ \vb=-\vb}}.
  \end{align*}
  By \eqref{eq:JT skew G}, \eqref{eq:dualJT skew G}, and \eqref{eq:omega_C=D}, we have 
  \begin{align*}
    \omega(G_{\lm}(\vx;\va,\vb)) = G_{\lambda'/\mu'}(\vx;\vb,\va)\Big\rvert_{\substack{\va = -\va \\ \vb=-\vb}} = G_{\lambda'/\mu'}(\vx;-\vb,-\va).
  \end{align*}
  Similarly, we obtain \( \omega(g_\lm(\vx;\va,\vb)) = g_{\lambda'/\mu'}(\vx;-\vb,-\va) \) using \eqref{eq:JT skew g}, \eqref{eq:dualJT skew g}, and
  \begin{align*}
    \omega(h_m[X-A_{\lambda_i-1}+A_{\mu_j} +B_{i-1} -B_{j-1}])= e_m[X-B_{i-1} +B_{j-1}+A_{\lambda_i-1}-A_{\mu_j}]\Big\rvert_{\substack{\va = -\va \\ \vb=-\vb}}.
  \end{align*}
\end{remark}

\section{Further study}
\label{sec:further_study}

In this paper we introduced refined canonical stable Grothendieck polynomials \(
G_{\lambda}(\vx;\va,\vb) \) and their duals \( g_{\lambda}(\vx;\va,\vb)\) with
two infinite sequences  \( \va \) and \( \vb \) of parameters. Since these generalize
Grothendieck polynomials \( G_\lambda(\vx) \) and their duals \( g_\lambda(\vx)
\), a natural question is to extend known properties of \( G_\lambda(\vx) \) and
\( g_\lambda(\vx) \) to \( G_{\lambda}(\vx;\va,\vb) \) and \(
g_{\lambda}(\vx;\va,\vb) \), respectively. In this section we propose some open
problems in this regard.

Recall that we have two combinatorial models for \( G_{\lambda}(\vx;\va,\vb) \)
in Theorem~\ref{thm:comb_intro} and Corollary~\ref{cor:G_comb2}, and for \(
g_{\lambda}(\vx;\va,\vb) \) in Theorem~\ref{thm:comb_intro} and
Corollary~\ref{cor:g_comb2}. If \( \va=(0,0,\dots) \) and \( \vb=(1,1,\dots) \),
then there is an algorithm, called an \emph{uncrowding algorithm}, which proves
the equivalence of these two combinatorial models for \( G_\lambda(\vx) \). The uncrowding algorithm was
first developed by Buch \cite[Theorem~6.11]{Buch2002} and 
generalized by Chan--Pflueger \cite{CP21:grothendieck}, Reiner--Tenner--Yong
\cite{RTY2018}, and Pan--Pappe--Poh--Schilling \cite{Pan2022}.
The equivalence of the two combinatorial models for \( g_\lambda(\vx) \)
was also shown bijectively by Lam and Pylyavskyy \cite{LP2007}.

\begin{problem}
  Find a combinatorial proof (desirably by using a modified uncrowding
  algorithm) of the equivalence of the two combinatorial models for \(
  G_{\lambda}(\vx;\va,\vb) \) in Theorem~\ref{thm:comb_intro} and
  Corollary~\ref{cor:G_comb2}. Similarly, for \( g_{\lambda}(\vx;\va,\vb) \),
  find a combinatorial proof of the equivalence of Theorem~\ref{thm:comb_intro} and
  Corollary~\ref{cor:g_comb2}.
\end{problem}

Buch \cite[Theorem~5.4]{Buch2002} found a combinatorial interpretation for the
Littlewood--Richardson-like coefficients \( c^\nu_{\lambda,\mu} \) for \( G_\lambda(\vx) \) defined by
\begin{equation}\label{eq:LR_G}
  G_\lambda(\vx)G_\mu(\vx)  = \sum_{\nu}c^\nu_{\lambda,\mu}G_\nu(\vx).
\end{equation}
Buch \cite[Corollary~6.7 and Theorem~6.9]{Buch2002} also found combinatorial interpretations for
different types of Littlewood--Richardson-like coefficients \( {d}^\nu_{\lambda,\mu}
\) and \( e^\nu_{\lambda,\mu} \) defined by
\begin{align}
  \label{eq:LR_Gd}
    \Delta G_{\nu}(\vx)  &= \sum_{\lambda,\mu}{d}^\nu_{\lambda,\mu}G_\lambda(\vx)\otimes G_\mu(\vx),\\
    \label{eq:LR_G'}
    G_{\nu/\lambda}(\vx)  &= \sum_{\mu}e^\nu_{\lambda,\mu}G_\mu(\vx),
\end{align}
where \( \Delta \) is the coproduct in the Hopf algebra of symmetric functions,
see \cite{grinberg14:hopf_algeb_combin}.
Yeliussizov \cite[Theorem~3.3 and (37)]{Yeliussizov2017} generalized
\eqref{eq:LR_G} and \eqref{eq:LR_Gd} to \( G_\lambda^{(\alpha,\beta)}(\vx) \).
As observed by Yeliussizov \cite[p.297]{Yeliussizov2017},
by the duality \( \langle G_\lambda(\vx), g_\mu(\vx) \rangle = \delta_{\lambda,\mu} \),
\eqref{eq:LR_Gd} is equivalent to
\[
  g_\lambda(\vx)g_\mu(\vx)  =
  \sum_{\nu}d^\nu_{\lambda,\mu} g_\nu(\vx) .
\]

Similarly, we define the Littlewood--Richardson-like coefficients \(
c^\nu_{\lambda,\mu} (\va,\vb)\), \( d^\nu_{\lambda,\mu} (\va,\vb)\) and \(
e^\nu_{\lambda,\mu} (\va,\vb) \) by
\begin{align*}
  G_\lambda(\vx;\va,\vb)G_\mu(\vx;\va,\vb)  &=
                                              \sum_{\nu}c^\nu_{\lambda,\mu}(\va,\vb) G_\nu(\vx;\va,\vb),\\
  \Delta G_{\nu}(\vx;\va,\vb)  &= \sum_{\lambda,\mu}{d}^\nu_{\lambda,\mu}(\va,\vb)G_\lambda(\vx;\va,\vb)\otimes G_\mu(\vx;\va,\vb),\\
  G_{\nu/\lambda}(\vx;\va,\vb)  &= \sum_{\mu}e^\nu_{\lambda,\mu}(\va,\vb) G_\mu(\vx;\va,\vb).
\end{align*}
Using the duality in Theorem~\ref{thm:vec dual}, the skew operator \( f^\perp
\), and the fact
\[
  \Delta g_\lambda(\vx;\va,\vb)=\sum_{\mu\subseteq\lambda}
  g_{\mu}(\vx;\va,\vb)\otimes g_{\lm}(\vx;\va,\vb),
\]
one can show that
\begin{align*}
 g_\lambda(\vx;\va,\vb) g_\mu(\vx;\va,\vb)
   &= \sum_\nu d^\nu_{\lambda,\mu}(\va,\vb) g_\nu(\vx;\va,\vb), \\
 \Delta g_\nu(\vx;\va,\vb)
   &= \sum_{\lambda,\mu} c^\nu_{\lambda,\mu}(\va,\vb) g_\lambda(\vx;\va,\vb)
          \otimes g_\mu(\vx;\va,\vb),~\mbox{and} \\
 g_{\nu/\lambda}(\vx;\va,\vb)
   &= \sum_\mu c^\nu_{\lambda,\mu}(\va,\vb) g_\mu(\vx;\va,\vb).
\end{align*}

\begin{problem}
  Find the Littlewood--Richardson-like coefficients \( c^\nu_{\lambda,\mu}
  (\va,\vb)\), \( d^\nu_{\lambda,\mu} (\va,\vb)\) and
  \( e^\nu_{\lambda,\mu} (\va,\vb) \).
\end{problem}

The theory of noncommutative Schur functions was developed by Fomin and Greene
\cite{FG1998} and further studied by Blasiak and Fomin \cite{Blasiak2017}. Buch
\cite[Proof of Theorem~3.2]{Buch2002} expressed the stable Grothendieck
polynomial \( G_\lm(\vx) \) using the Fomin--Greene operators. Similar
expressions were given by Lam and Pylyavskyy \cite[Proof of Theorem~9.1]{LP2007}
for \( g_\lm(\vx) \) and by Yeliussizov \cite[Theorem~5.3]{Yeliussizov2017}
for \( G_\lambda^{(\alpha,\beta)}(\vx) \). Yeliussizov
\cite[Section~5.4]{Yeliussizov2017} also found a similar expression for \(
G_{\lambda/\!\!/\mu}(\vx) \), which is another skew version of \(
G_{\lambda}(\vx) \) different from \( G_{\lm}(\vx) \).

\begin{problem}
  Find expressions for \( G_{\lambda}(\vx;\va,\vb) \) and \(
  g_{\lambda}(\vx;\va,\vb) \) (or their skew versions) using Fomin--Greene-type
  operators.
\end{problem}

Finally, as mentioned in the introduction it is known that Grothendieck polynomials
\( G_{\lambda}(\vx) \) and \( g_{\lambda}(\vx) \) have integrable vertex
models~\cite{Brubaker2023, Buciumas2020, Gunna20,
MS13,Motegi2021,WZ19}, crystal structures~\cite{Galashin2017,HS20,Pan2022} and
probabilistic models~\cite{Motegi2021, yeliussizov20:_dual_groth}. It would
be very interesting to extend these results to \( G_{\lambda}(\vx;\va,\vb) \) and
\( g_{\lambda}(\vx;\va,\vb) \).

\section*{Acknowledgments}

The authors would like to thank the anonymous referee for extremely
careful reading of the manuscript and for providing plenty of useful
comments and suggestions. They are also grateful to Travis Scrimshaw
and Darij Grinberg for many helpful comments and fruitful discussions.

\bibliographystyle{abbrv}

\end{document}